\documentclass[reqno]{amsart}
\usepackage{amssymb}
\usepackage[usenames, dvipsnames]{color}

\usepackage{hyperref}

\usepackage{verbatim}

\usepackage{pxfonts}

\numberwithin{equation}{section}

\newtheorem{theorem}{Theorem}[section]
\newtheorem{corollary}[theorem]{Corollary}
\newtheorem{lemma}[theorem]{Lemma}
\newtheorem{proposition}[theorem]{Proposition}

\theoremstyle{definition}
\newtheorem{remark}[theorem]{Remark}

\theoremstyle{definition}

\theoremstyle{definition}
\newtheorem{assumption}[theorem]{Assumption}

\newcommand{\bR}{\mathbb{R}}

\newcommand\cH{\mathcal{H}}

\newcommand\cX{\mathcal{X}}

\newcommand\Div{\operatorname{div}}

\providecommand{\set}[1]{\{#1\}}

\providecommand{\norm}[1]{\lVert#1\rVert}

\begin{document}
\title[Stokes systems with VMO coefficients]{Boundary Lebesgue mixed-norm estimates for non-stationary Stokes systems with VMO coefficients}

\author[H. Dong, D. Kim, and T. Phan]{Hongjie Dong, Doyoon Kim, and Tuoc Phan}

\address[H. Dong]{Division of Applied Mathematics, Brown University,
182 George Street, Providence, RI 02912, USA}
\email{\href{mailto:Hongjie\_Dong@brown.edu}{\nolinkurl{Hongjie\_Dong@brown.edu}}}
\address[D. Kim]{Department of Mathematics, Korea University, 145 Anam-ro, Seongbuk-gu, Seoul 02841, Republic of Korea}
\email{\href{mailto:doyoon_kim@korea.ac.kr}{\nolinkurl{doyoon_kim@korea.ac.kr}}}
\address[T. Phan]{Department of Mathematics, University of Tennessee, 227 Ayres Hall, 1403 Circle Drive, Knoxville, TN 37996, USA}
\email{phan@math.utk.edu}
\thanks{H. Dong was partially supported by the NSF under agreement DMS-1600593; D. Kim was supported by the National Research Foundation of Korea (NRF) grant funded by the Korea government (MSIT) (2019R1A2C1084683); T. Phan is partially supported by the Simons Foundation, grant \#354889.}



\begin{abstract}
We consider Stokes systems with measurable coefficients and Lions-type boundary conditions.
We show that, in contrast to the Dirichlet boundary conditions, local boundary mixed-norm $L_{s,q}$-estimates hold for the spatial second-order derivatives of solutions, assuming the smallness of the mean oscillations  of the coefficients with respect to the spatial variables in small cylinders.
In the un-mixed norm case with $s=q=2$, the result is still new and provides local boundary Caccioppoli-type estimates.
The main challenges in the work arise from the lack of regularity of the pressure and time derivatives of the solutions and from interaction of the boundary with the nonlocal structure of the system. To overcome these difficulties, our approach relies heavily on several newly developed regularity estimates for both divergence and non-divergence form parabolic equations with coefficients that are only measurable in the time variable and in one of the spatial variables.
\end{abstract}

\maketitle

\section{Introduction and main results}

In this paper, we investigate local boundary mixed-norm $L_{s,q}$-estimates for solutions to time-dependent Stokes systems.
In particular, we show that for time-dependent Stokes systems with the Lions boundary conditions (see  \cite{Li69,MR1422251} and \eqref{bdr-cond} below), the local boundary $L_{s,q}$-estimates for the solutions hold. Our results are established for a general class of Stokes systems in non-divergence form with measurable coefficients, so they could therefore be useful, for example, for studying flows of inhomogeneous fluids with density-dependent viscosity \cite{CKST, MR2092995}.  Precisely, we investigate the following Stokes system:
\begin{equation}  \label{eq7.41c}
                      u_t-a_{ij}(t, x)D_{ij} u+\nabla p=f,\quad \Div u= g \quad \text{in} \  Q_1^+,
\end{equation}
with the Lions boundary conditions on $\{x_d= 0\}$:
\begin{equation} \label{bdr-cond}
D_d u_k = u_d = 0 \quad \text{on} \  (-1, 0] \times B_1' \times \{0\}, \quad k=1,2,\ldots, d-1.
\end{equation}
The Lions boundary conditions are a special case of the Navier (or slip) boundary conditions introduced in \cite{Na27}.
In the above equations, $Q_1^+$ is the unit upper half-parabolic cylinder and $B_1'$ is the unit ball in $\bR^{d-1}$. See Section \ref{sec2.1} for their definitions. In \eqref{eq7.41c},
 $$u= (u_1(t,x), u_2(t,x), \ldots, u_d(t,x))  \in \bR^d, \quad \text{where} \quad (t,x) \in  Q_1^+,$$
is an unknown vector-valued function representing the velocity of the considered fluid,  $p = p(t,x)$ is an unknown fluid pressure,  $f = (f_1, f_2, \ldots, f_d)$ is a given measurable vector-valued function, and $g = g(t,x)$ is a given measurable function.
In addition, $a_{ij} = a_{ij}(t,x)$ is a given measurable symmetric matrix of the viscosity coefficients. Throughout the paper, we assume that  $a_{ij}$ satisfies the following boundedness and ellipticity conditions with the ellipticity constant $\nu \in (0,1)$: for a.e. $(t,x)$,
\begin{equation} \label{ellipticity}
\nu |\xi|^2 \leq a_{ij}(t,x) \xi_i \xi_j \quad \text{and} \quad |a_{ij}| \leq \nu^{-1} \quad \text{for} \  \xi = (\xi_1, \xi_2, \ldots, \xi_d) \in \bR^d.
\end{equation}
 As a regularity assumption on the coefficients,  we impose the following vanishing mean oscillation in $x$ (VMO$_x$) condition on $a_{ij}$, which was introduced in \cite{MR2304157}, with a constant $\delta \in (0,1)$.
\begin{assumption}[$\delta$]
                        \label{assump1}
There exists $R_0\in (0, 1/4)$ such that for any $(t_0,x_0)\in \overline{Q_2^+}$ and $r\in (0,R_0)$, there exists $\hat{a}_{ij}(t)$ satisfying \eqref{ellipticity} and
$$
\fint_{Q_r^+(t_0,x_0)} |a_{ij}(t,x)-\hat{a}_{ij}(t)|\,dx\,dt\le \delta \quad \text{for} \  i, j = 1,2,\ldots, d.
$$
\end{assumption}

For the definitions of $Q_r^+(t_0,x_0)$ and various function spaces, we refer the reader to Section \ref{sec2.1}. We say that $(u,p) \in  W_1^{1,2}(Q_1^+)^d \times W_1^{0,1}(Q_1^+)$ is a strong solution of \eqref{eq7.41c} on $Q_1^+$ if \eqref{eq7.41c} holds for a.e. $(t,x) \in Q_1^+$ and \eqref{bdr-cond} holds in the sense of trace.
The main result of the paper on the local $L_{s,q}$-estimate for solutions to \eqref{eq7.41c} is now stated as the following theorem.

\begin{theorem}
                                \label{thm2.3b}
Let $s, q\in (1,\infty)$.
There exists $\delta=\delta(d,\nu,s,q) \in (0,1)$ such that the following statement holds. Suppose that Assumption \ref{assump1} $(\delta)$ holds.
Then, if $(u,p) \in W^{1,2}_{s,q}(Q_1^+)^d \times W_1^{0,1}(Q_1^+)$ is a strong solution to \eqref{eq7.41c} in $Q_1^+$ with the boundary conditions \eqref{bdr-cond}, $f \in   L_{s,q}(Q_1^+)^d$, and $D g \in L_{s,q}(Q_1^+)^d$, it follows that
\begin{equation} \label{main-est-2}
\begin{split}
\|D^2u\|_{L_{s, q}(Q_{1/2}^+)} & \le N(d,\nu, s, q)\Big[ \|f\|_{L_{s,q}(Q_{1}^+)} +  \|Dg\|_{L_{s,q}(Q_{1}^+)}\Big]\\
& \quad +N(d,\nu,s, q) R_0^{-2}\|u\|_{L_{s,q}(Q_{1}^+)}.
\end{split}
\end{equation}
\end{theorem}

\begin{remark}

(i) By using interpolation and a standard iteration argument, it is easily shown that \eqref{main-est-2} still holds if we replace the term $R_0^{-2}\norm{u}_{L_{s,q}(Q_1^+)}$ on the right-hand side with $R_0^{-2-d+d/q}\norm{u}_{L_{s,1}(Q_1^+)}$.

(ii) The estimate \eqref{main-est-2} holds trivially for $d=1$. Therefore, throughout the paper, we set $d \geq 2$.
\end{remark}

Even in the un-mixed norm case with $s=q=2$, the estimate \eqref{main-est-2} is new.  In this case, local boundary estimates as in \eqref{main-est-2} are known as Caccioppoli-type estimates. See \cite{MR2429247, arXiv:1805.04143, MR3116962, MR2646528}, for instance. However, in contrast to the case we consider, the local boundary Caccioppoli-type estimates for non-stationary Stokes systems do not hold under the homogeneous Dirichlet boundary conditions, as demonstrated in a recent work \cite{arXiv:1806.02516}.
Therefore,  besides other interests, finding a right class of boundary conditions so that \eqref{main-est-2} holds is an interesting question, which this paper answers.

We emphasize that the boundary conditions \eqref{bdr-cond} are essential to the validity of \eqref{main-est-2}.
Observe that unlike some known local regularity estimates (see \cite{MR3289443}, for instance), \eqref{main-est-2} does not contain the pressure on the right-hand side, and thus it requires only very mild regularity of the pressure. To the best of our knowledge, it is new even for the classical Stokes system, i.e., when $a_{ij}=\delta_{ij}$. As such, \eqref{main-est-2} might be useful in applications.
For more information regarding estimates without the pressure, see \cite[Remark IV.4.2]{MR2808162} and \cite{MR2027755, MR3822765, MR1789922} for stationary equations with constant coefficients, \cite{MR2429247} for time-dependent equations with constant coefficients, and \cite{arXiv:1805.04143} for time-dependent equations with measurable coefficients.

The $L_q$-estimates for solutions of Stokes systems are a research topic of great mathematical interest. See the monographs \cite{MR2808162, MR3289443, MR3822765}, as well as a survey paper \cite{HS18} and the references therein. The earliest work on equations with constant coefficients can be found in \cite{MR0171094}. See also \cite{GS91, MR1359996, MR1923551}. In these works, global estimates are proved either using fundamental solutions and potential analysis techniques, or using a functional analytic approach. Local estimates are more delicate and cannot be derived from these methods.
In recent work \cite{MR3758532, MR3949432}, the local and global $L_q$ and weighted $L_q$ theory are established for divergence form stationary Stokes systems with measurable coefficients using a perturbation method and localization technique.
However, this approach does not work for non-stationary Stokes systems owing to the lack of local regularity in the time variable of solutions and the pressure.
This problem is considered in a recent work \cite{arXiv:1805.04143}, in which local interior estimates in mixed-norm Lebesgue spaces are established by combining the perturbation argument with several regularity estimates for equations in divergence and non-divergence form applied to the vorticity equations. In this paper, we study the corresponding local boundary estimates.

The proof of Theorem \ref{thm2.3b} is based on the perturbation technique using the Fefferman--Stein sharp functions developed in \cite{MR2550072, MR2304157, MR2435520} and in \cite{MR3758532, MR3949432, arXiv:1805.04143}. There are several additional difficulties.
First, as we already mentioned, the localization technique typically used in the study of stationary Stokes systems \cite{MR3758532, MR3949432} is not applicable owing to the lack of regularity in the time variable for the Stokes system.
Second, the structure of the system is nonlocal in view of the pressure term, and its complicated interaction with the boundary is not very well understood.
Finally, the usual local energy estimates that are essential in perturbation methods are not known in the literature for the time-dependent Stokes system \eqref{eq7.41c}.
To overcome these difficulties, we modify the ideas used in \cite{arXiv:1805.04143} and take the boundary conditions \eqref{bdr-cond} into account to derive boundary estimates for the solutions of the vorticity equations and divergence equations. Several new intermediate results on the solvability and regularity estimates for the Stokes system and the vorticity equations near the boundary are  developed.

In the rest of this section, we briefly discuss a result on the solvability of the Stokes system with the Lions boundary conditions. The result is not only intrinsically interesting, but is also an essential ingredient that we develop to prove Theorem \ref{thm2.3b}.
Consider the following Stokes system in the upper half-space:
\begin{equation} \label{u-sol.eqn-whole}
\left\{
\begin{array}{cccl}
u_t -a_{ij}(t) D_{ij} u + \nabla p & = & f  & \quad \text{in} \   (0, T] \times \bR_{+}^d, \\
\Div u & = & g & \quad \text{in} \  (0, T] \times  \bR_{+}^d,\\
u(0,x) & =& 0& \quad \text{for} \ x \in \bR^d_+,
\end{array} \right.
\end{equation}
with the Lions boundary conditions
\begin{equation} \label{Lions-bdr-c}
D_du_k = u_d  =  0  \quad \text{on} \  (0, T] \times  \bR^{d-1} \times \{0\} \quad \text{for} \ k = 1,2,\ldots, d-1,
\end{equation}
where $T>0$ is some given number and $\mathbb{R}^d_+ = \mathbb{R}^{d-1} \times (0, \infty)$. In \eqref{u-sol.eqn-whole}, we assume that $a_{ij}$ is a measurable function depending only on the time variable, i.e., $a_{ij}: (0, T) \rightarrow \bR$, and that \eqref{ellipticity} holds.

\begin{theorem} \label{existence-lemma}
Let $T >0$ and $q_0 \in (1, \infty)$. Let $f \in L_{q_0}((0, T) \times \bR_{+}^d)^d$ and $g: (0, T) \times \bR_{+}^d \rightarrow \bR$ such that $g\in L_{q_0}((0, T) \times \bR_{+}^d)$, $Dg \in L_{q_0}((0, T) \times \bR_{+}^d)^d$, $g(0,\cdot) =0$, and $g_t = \Div G$ for some vector field
$$
G = (G_1, G_2,\ldots, G_d) \in L_{q_0}((0, T) \times \bR_{+}^d)^d
$$
in the sense that
\begin{equation} \label{eq0702_01}
\int_{(0, T)\times \bR^d_{+}} g \varphi_t \, dx \, dt = \int_{(0, T) \times \bR^d_{+}} G_i D_i \varphi \, dx \, dt
\end{equation}
for any $\varphi \in C_0^\infty((0, T) \times \bR^d)$.
Then there exists a unique strong solution $(u, p)$ of \eqref{u-sol.eqn-whole}--\eqref{Lions-bdr-c} such that
\[
\begin{split}
& u \in L_\infty\left((0,T), L_{q_0}(\bR^d_+)\right), \quad u_t,  Du, D^2 u \in L_{q_0}((0, T) \times \bR_{+}^d),\\
& p \in L_{q_0}((0, T), L_{q_0, \textup{loc}}(\bR^d_+)), \quad \nabla p  \in L_{q_0}((0, T) \times \bR_{+}^d ).
\end{split}
\]
Moreover,  $(u,p)$ satisfies the estimates
\begin{equation}\label{eq0715_07}
\begin{split}
& \norm{Du}_{L_{q_0}((0, T) \times \bR_{+}^d )} \leq N_1 \norm{f}_{L_{q_0}((0, T) \times \bR_{+}^d)} + N_2 \norm{g}_{L_{q_0} ((0, T) \times \bR_{+}^d)},
\\
& \norm{D^2u}_{L_{q_0}((0, T) \times \bR_{+}^d)}  \leq N_2 \left[ \norm{f}_{L_{q_0}((0, T) \times \bR_{+}^d)} + \norm{D g}_{L_{q_0} ((0, T) \times \bR_{+}^d)} \right],
\\
& \norm{\nabla p}_{L_{q_0}((0, T) \times \bR_{+}^d)} \leq N_2 \left[ \norm{f}_{L_{q_0}((0, T) \times \bR_{+}^d)} + \norm{D g}_{L_{q_0} ((0, T) \times \bR_{+}^d)} + \norm{G}_{L_{q_0}((0, T) \times \bR_{+}^d)} \right],
\end{split}
\end{equation}
and
\begin{equation}
\label{eq0715_09}
\norm{u_t}_{L_{q_0}((0, T) \times \bR_{+}^d)} \leq N_2 \left[  \norm{f}_{L_{q_0}((0, T) \times \bR_{+}^d)} + \norm{G}_{L_{q_0}((0, T) \times \bR_{+}^d)} \right],
\end{equation}
for some constants $N_1 = N_1(\nu, d, q_0, T) >0$ and $N_2 = N_2(\nu, d, q_0) >0$.
\end{theorem}

Although the Stokes system with the Lions boundary conditions appeared some time ago \cite{Li69, MR1422251}, Theorem \ref{existence-lemma} seems new.
To carry out the proof, we use the boundary conditions and carefully use odd/even extensions to look for a solution in the whole space.
To avoid the complication due to the pressure, we first solve for the vorticity, from which we recover the solution using the divergence equation and the fundamental solution of the Laplace equation. Because of the odd and even extensions, the new coefficients of the Stokes system in the whole space are merely measurable with possibly very large oscillation in the $x_d$ direction.
Therefore, solving and estimating the solutions in Sobolev spaces are quite involved. Several recent results developed in \cite{MR2833589, MR2550072} on the existence, uniqueness, and regularity results for equations with coefficients only measurable in $t$ and one of the spatial directions are carefully applied to obtain the desired results.

The rest of the paper is organized as follows. In Section \ref{Preli}, we introduce the notation and recall several known inequalities and estimates that are needed in the paper. In Section \ref{constant-sec}, we study the Stokes system with coefficients depending only on the time variable.
Several regularity estimates of solutions near the boundary are proved using the divergence and vorticity equations.
In Section \ref{sec4}, we prove Theorem \ref{existence-lemma} on the existence and uniqueness of strong solutions to the Stokes system in the upper half-space with the Lions-type boundary conditions.
In the last section, Section \ref{non-div-se}, Theorem \ref{thm2.3b} is proved.

\section{Notation and preliminary estimates} \label{Preli}

\subsection{Notation}           \label{sec2.1}
We denote the upper half-ball in $\mathbb{R}^d$ of radius $\rho$ centered at $x_0 = (x'_0, x_{d0}) \in \bR^{d-1} \times \mathbb{R}$ as
\[
B_\rho^+(x_0) =\{x = (x', x_d) \in \bR^{d-1} \times \bR: |x-x_0| <\rho, \ x_d > 0 \}
\]
and the upper half-parabolic cylinder centered at $z_0 = (t_0, x_0) \in \bR^{d+1}$ with radius $\rho>0$ as
\[
Q_\rho^+(z_0) = (t_0 - \rho^2, t_0] \times B_\rho^+(x_0).
\]
For brevity, when $z_0 = (0,0)$, we write $Q_\rho^+ = Q_\rho^+(0,0)$ and $B_\rho^+ = B_\rho^+(0)$. We also denote by $B_\rho'$ the unit ball in $\bR^{d-1}$ centered at the origin with radius $\rho>0$.

For each $s, q \in [1, \infty)$ and each parabolic cylinder $Q  = \Gamma \times \Omega \subset \bR \times \bR^{d}$, the Lebesgue mixed $(s,q)$-norm of a measurable function $h$ defined in $Q$ is
\[
\norm{h}_{L_{s,q}(Q)} = \left[\int_{\Gamma} \left( \int_{\Omega} |h(t,x)|^q \,dx \right)^{s/q} dt \right]^{1/s},
\]
and we denote the mixed-norm Lebesgue spaces as
\[
L_{s,q}(Q)= \{h : Q \rightarrow \mathbb{R}: \|h\|_{L_{s,q}(Q)} <\infty\}.
\]
We also denote the parabolic Sobolev space as
\begin{align*}
W_{s,q}^{1,2}(Q)&=
\set{u:\,u, Du,D^2u\in L_{s,q}(Q), \ u_t \in L_{1}(Q)},
\end{align*}
which is slightly different from the usual parabolic Sobolev spaces as it does not require $u_t \in L_{{s,q}}(Q)$. We also set
$$
W_{s,q}^{0,1}(Q) = \set{u:\,u, Du\in L_{s,q}(Q)}.
$$
When $s= q$, we omit one of these two indices and write
\[
L_q(Q) = L_{q,q}(Q), \quad W_{q}^{1,2}(Q) = W_{q,q}^{1,2}(Q) , \quad W_{q,q}^{0,1}(Q) = W_q^{0,1}(Q).
\]

\subsection{Sharp function estimates}
The following result is a special case of \cite[Theorem 2.3 (i)]{MR3812104}.
Let $\cX\subset \bR^{d+1}$ be a space of homogeneous type, which is endowed with the parabolic distance and a doubling measure $\mu$ that is naturally inherited from the Lebesgue measure.
As in \cite{MR3812104}, we take a filtration of partitions of $\cX$ (cf. \cite{MR1096400}) and, for any $f\in L_{1,\text{loc}}$, we define its dyadic sharp function $f^{\#}_{\text{dy}}$ in $\cX$ associated with the filtration of partitions.
In addition, for each $q \in [1, \infty]$, $A_q$ denotes the Muckenhoupt class of weights.

\begin{theorem} \label{Feffer}
Let $s, q \in (1,\infty)$, $K_0 \geq 1$, and $\omega \in A_{q}$ with $[\omega]_{A_q} \leq K_0$. Suppose that $f \in L_{s}(\omega d\mu)$.
Then,
\[
\norm{f}_{L_s(\omega d\mu)} \leq  N\left[\norm{f^{\#}_{\text{dy}}}_{L_{s}(\omega d\mu)} + \mu(\mathcal{X})^{-1} \omega(\textup{supp}(f))^{\frac{1}{s}} \norm{f}_{L_1(\mu)} \right],
\]
where $N>0$ is a constant depending only on $s$, $q$, $K_0$, and the doubling constant of $\mu$, and the second term on the right-hand side is understood to be zero if $\mu(\cX) = \infty$.
\end{theorem}

As a direct consequence of Theorem \ref{Feffer}, we have the following lemma, where $f^{\#}_{\textup{dy}}$ is the dyadic sharp function of $f$ on $Q_R^+$ associated with a filtration of partitions of $Q_R^+$ satisfying the properties in, for instance, \cite[Theorem 2.1]{MR3812104}.
Note that, as for $Q_R^+$, the constants in \cite[Theorem 2.1]{MR3812104} depend only on the dimension $d$.

\begin{lemma}
    \label{mixed-norm-lemma}
For any $s,q \in (1, \infty)$, there exists a constant $N = N(d, s, q) >0$ such that
\[
\norm{f}_{L_{s,q}(Q^+_R)} \leq N\Big[ \norm{f^{\#}_{\textup{dy}}}_{L_{s,q}(Q_R^+)} +   R^{\frac{2}{s}+\frac {d}{q} - d-2} \norm{f}_{L_1(Q_R^+)}\Big]
\]
for any $R>0$ and $f \in L_{s,q}(Q_R^+)$.
\end{lemma}

\begin{proof}
For $t\in (-R^2,0)$, let
\[
\psi(t) = \norm{f(t, \cdot)}_{L_{q}(B_R^+)} \quad \text{and} \quad
\phi(t) = \norm{f^{\#}_{\textup{dy}}(t,\cdot)+(|f|)_{Q_R^+}}_{L_{q}(B_R^+)}.
\]
Moreover, for any $\omega \in A_{q}((-R^2,0))$ with $[\omega]_{A_{q}} \leq K_0$, we write $\tilde{\omega}(t,x) = \omega(t)$ for all $(t,x) \in Q_R^+$.
Then, by applying Theorem \ref{Feffer} with $\cX=Q_R^+$, we obtain
\[
\norm{\psi}_{L_{q}((-R^2,0), \omega)}= \norm{f}_{L_{q}(Q_R^+, \tilde{\omega})}\le N\norm{f^\#_{\text{dy}}+(|f|)_{Q_R^+}}_{L_{q}(Q_R^+, \tilde{\omega})}
=  N \norm{\phi}_{L_{q}((-R^2,0), \omega)},
\]
where $N = N(d, K_0, s)$.
Then, by the extrapolation theorem (see, for instance, \cite[Theorem 2.5]{MR3812104}), we see that
\[
\norm{\psi}_{L_{s}((-R^2,0), \omega)} \leq 4N \norm{\phi}_{L_{s}((-R^2,0), \omega)}, \quad \forall \ \omega \in A_{s}, \quad [\omega]_{A_{s}} \leq K_0.
\]
Note that in the special case of $\omega\equiv 1$, $\norm{\psi}_{L_{s}((-R^2,0), \omega)}  = \norm{f}_{L_{s,q}(Q_R^+)}$ and
\[
 \norm{\phi}_{L_{s}((-R^2,0), \omega)} \le \norm{f^{\#}_{\text{dy}}}_{L_{s,q}(Q_R^+)} + R^{\frac 2 s+\frac d q} (|f|)_{Q_R^+}.
\]
Therefore, the desired estimate follows.
\end{proof}
\section{Stokes systems with simple coefficients} \label{constant-sec}
In this section, we consider the time-dependent Stokes system with coefficients depending only on the time variable:
\begin{equation}  \label{eq7.41}
u_t- a_{ij}(t) D_{ij} u +\nabla p=0,\quad \Div u= 0 \quad \text{in} \,\, Q_1^+.
\end{equation}
The system \eqref{eq7.41} is equipped with the Lions boundary conditions on $\{x_d =0\} \cap B_1$: for $k =1,2,\ldots, d-1$,
\begin{equation} \label{simple-bd-cond}
D_d u_k (t, x', 0) = u_d (t, x', 0) =0 \quad \text{for a.e.} \,\, (t,x') \in (-1, 0] \times B_1',
\end{equation}
where $a_{ij}=  (a_{ij}(t))$ is a given symmetric matrix of coefficients depending only on the time variable $t$ and satisfying the ellipticity condition \eqref{ellipticity}.
This section provides key estimates that are needed for the proof of Theorem  \ref{thm2.3b}. We begin with the following estimates of the gradient and the second derivatives of solutions.

\begin{lemma}  \label{lem1.2}
Let $q_0\in (1,\infty)$, and $(u,p)\in W^{1,2}_{q_0}(Q_1^+)^d\times  W_1^{0,1}(Q_1^+)$ be a strong solution to \eqref{eq7.41} in $Q_1^+$ with the boundary conditions \eqref{simple-bd-cond}. Then we have
\begin{align}
                                        \label{eq7.46b}
\|D^2u\|_{L_{q_0}(Q_{1/2}^+)} & \le N(d,\nu,q_0)
\sum_{i=1}^{d-1}\Big(\|D_{x'}u_i-[D_{x'}u_i]_{B_1^+}(t)\|_{L_{q_0}(Q_{1}^+)}
+\|D_{d}u_i\|_{L_{q_0}(Q_{1}^+)}\Big)\nonumber\\
&\quad+N(d,\nu,q_0) \|D_{x'}u_d\|_{L_{q_0}(Q_{1}^+)},
\end{align}
where $[Du_i]_{B_1^+}(t)$ is the average of $Du_i(t,\cdot)$ in $B_1^+$.
Moreover, \eqref{eq7.46b} also holds if the second equation in \eqref{eq7.41} is replaced by
\[
\Div u = g(t) \quad \text{in} \  Q_1^+
\]
for some given measurable function $g: (-1,0) \rightarrow \mathbb{R}$, which is independent of the spatial variables.
\end{lemma}

\begin{proof}
We prove \eqref{eq7.46b} when the second equation in \eqref{eq7.41} is replaced by
\[
\Div u = g(t) \quad \text{in} \  Q_1^+.
\]
Let $(\omega_{kl})_{k, l =1}^d$ be a matrix-valued function defined in $Q_1^+$ as
\begin{equation} \label{omega-def}
\omega_{kl} = \partial_k u_l - \partial_l u_k \quad \text{in} \  Q_1^+, \quad \text{for} \ k, l \in \{1, 2,\ldots, d\}.
\end{equation}
Then $\omega_{kl} \in \cH_{q_0}^1(Q_1^+)$ for $k,l \in \{1,2,\ldots,d\}$. See \cite[p. 362]{MR2764911} for the definition of $\cH_{q_0}^1(Q_1^+)$.
We set
$$
\tilde{a}_{ij}(t) = a_{ij}(t), \quad \tilde{a}_{dd} = a_{dd}(t)
$$
for $i,j \in \{1,\ldots,d-1\}$ and
$$
\tilde{a}_{dj}(t) = 2a_{dj}(t), \quad
\tilde{a}_{jd} = 0
$$
for $j=1,\ldots,d-1$.
We observe that for every $k, l \in \{1, 2,\ldots, d\}$, $\omega_{kl}\in \cH^1_{q_0}{ (Q_1^+)}$ is a weak solution to the parabolic equation
\begin{equation}
                        \label{eq8.05}
\partial_t \omega_{kl}- \Div \left( \tilde{a} (t)^{\top} \nabla \omega_{kl} \right) =0  \quad \text{in} \  Q_1^+
\end{equation}
with the homogeneous Dirichlet or homogeneous conormal derivative boundary condition on $\{x_d =0\}$. Precisely,
\begin{equation} \label{bdr-vorticity}
\left\{
\begin{array}{ll}
 \displaystyle \sum_{j=1}^d \tilde{a}_{jd}(t) D_j \omega_{k l} = a_{dd}(t) D_d \omega_{kl} = 0 \quad \text{on} \ \{x_d =0\} \cap B_1 & \ \ \text{if} \, k, l \in \{1, 2,\ldots, d-1\},\\
 \omega_{kl} =0  \quad \text{on} \ \{x_d =0\} \cap B_1 & \ \ \text{if} \, k =d \ \text{or} \ l =d.
\end{array} \right.
\end{equation}
From this, we apply the local boundary $\cH^1_p$-estimate for linear parabolic equations in divergence form (cf. \cite{MR2764911}) to obtain
\begin{equation}
                            \label{eq7.57}
\|D\omega\|_{L_{q_0}(Q_{2/3}^+)}\le N(d,\nu,q_0)\|\omega\|_{L_{q_0}(Q_{3/4}^+)}.
\end{equation}
Since $\Div u = g(t)$ and $g$ is independent of the $x$-variable, we have
\begin{equation} \label{ui-eqn}
\Delta u_i=  -D_i  \sum_{k=1}^d D_k u_k + \sum_{k= 1 }^d D_{kk}u_i
=  \sum_{k=1}^d D_k \omega_{ki} \quad \text{a.e. in}  \quad Q_1^+.
\end{equation}
Then, upon using the boundary conditions \eqref{simple-bd-cond} and \eqref{bdr-vorticity}, for a.e. $t \in (-1,0)$, one can view \eqref{ui-eqn} as the following Poisson equations in non-divergence form with the Neumann and Dirichlet boundary condition, respectively. Precisely, for a.e. $t \in (-1,0)$, the function $u_i = u_i(t,\cdot)$ satisfies
\[
\left\{
\begin{aligned}
\Delta \left(u_i - [u_i]_{B_1^+}(t)\right)=\sum_{k=1}^d D_k \omega_{ki} & \quad \text{in} \  B_1^+, \\
 D_d \left(u_i - [u_i]_{B_1^+}(t)\right) =  0 & \quad \text{on} \ \{x_d = 0\} \cap B_1
\end{aligned}
\right.
 \]
for $i=1,\ldots,d-1$, where $[u]_{B_1^+}(t)$ is the average of $u(t,\cdot)$ in $B_1^+$, and
\[
\left\{
\begin{aligned}
\Delta u_d =  \sum_{k=1}^d D_k \omega_{kd} & \quad \text{in} \  B_1^+, \\
  u_d = 0 & \quad \text{on} \ \{x_d = 0 \} \cap B_1.
\end{aligned} \right.
\]
We apply the local boundary $W_p^2$-estimate for the Laplace operator and then integrate it over the time variable to obtain
$$
\|D^2 u\|_{L_{q_0}(Q_{1/2}^+)} \le N \|D\omega\|_{L_{q_0}(Q_{2/3}^+)}+ N\sum_{i=1}^{d-1}\|u_i - [u_i]_{B_1^+} \|_{L_{q_0}(Q_{2/3}^+)} + N\|u_d\|_{L_{q_0}(Q_{2/3}^+)}.
$$
From this inequality and \eqref{eq7.57} we obtain that
\begin{align*}
&\|D^2 u\|_{L_{q_0}(Q_{1/2}^+)}
\le N \|\omega\|_{L_{q_0}(Q_{3/4}^+)}+ N\sum_{i=1}^{d-1}\|u_i - [u_i]_{B_1^+} \|_{L_{q_0}(Q_{2/3}^+)} + N\|u_d\|_{L_{q_0}(Q_{2/3}^+)}\\
&\le \varepsilon\|D^2u\|_{L_{q_0}(Q_{3/4}^+)}+N\varepsilon^{-1}\sum_{i=1}^{d-1}\|u_i-[u_i]_{B_1^+}(t)\|_{L_{q_0}(Q_{3/4}^+)} + N\varepsilon^{-1} \|u_d\|_{L_{q_0}(Q_{3/4}^+)},
\end{align*}
where in the second inequality we used multiplicative inequalities.
It then follows from a standard iteration argument that
$$
\|D^2u\|_{L_{q_0}(Q_{1/2}^+)}
\le N\sum_{i=1}^{d-1}\|u_i-[u_i]_{ B_1^+}(t)\|_{L_{q_0}(Q_{1}^+)} + N\|u_d\|_{L_{q_0}(Q_{1}^+)}.
$$
By the multiplicative inequalities again, we arrive at
\begin{equation}
							\label{eq1230_01}
\|Du\|_{L_{q_0}(Q_{1/2}^+)}
\le N\sum_{i=1}^{d-1}\|u_i-[u_i]_{ B_1^+}(t)\|_{L_{q_0}(Q_{1}^+)} + N\|u_d\|_{L_{q_0}(Q_{1}^+)}.
\end{equation}

Now by using the method of finite-difference quotient in the $x'$ direction and taking the limit, from \eqref{eq1230_01}, we get
\begin{equation}
                    \label{eq10.03}
\|DD_{x'}u\|_{L_{q_0}(Q_{1/2}^+)}\le
N \sum_{i=1}^{d-1}\|D_{x'}u_i-[D_{x'}u_i]_{ B_1^+}(t)\|_{L_{q_0}(Q_{1}^+)} + N\|D_{x'} u_d\|_{L_{q_0}(Q_{1}^+)},
\end{equation}
where $N = N(d, \nu,q_0)$.
Using the condition that $\Div u$ is independent of $x$, we also have
\begin{equation}
                    \label{eq10.05}
\|D_d^2 u_d\|_{L_{q_0}(Q_{1/2}^+)}\le N\sum_{i=1}^{d-1}\|D_{x'}u_i-[D_{x'}u_i]_{ B_1^+}(t)\|_{L_{q_0}(Q_{1}^+)} + N\|D_{x'}u_d\|_{L_{q_0}(Q_{1}^+)}.
\end{equation}
It remains to estimate $D_d^2 u_i$ for $i=1,\ldots,d-1$. Since
$$
D_d^2 u_i=D_d\omega_{di}+D_dD_i u_d,
$$
it follows from \eqref{eq7.57} and \eqref{eq10.03} that

\begin{equation}
                    \label{eq10.20}
\|D_d^2 u_i\|_{L_{q_0}(Q_{1/2}^+)}\le N \sum_{i=1}^{d-1}\left(\|D_{x'}u_i-[D_{x'}u_i]_{ B_1^+}(t)\|_{L_{q_0}(Q_{1}^+)} + \|\omega_{di}\|_{L_{q_0}(Q_{1}^+)}\right) + N\|D_{x'} u_d\|_{L_{q_0}(Q_{1}^+)}.
\end{equation}
Combining \eqref{eq10.03}, \eqref{eq10.05}, \eqref{eq10.20}, and the triangle inequality, we obtain \eqref{eq7.46b}.  The lemma is proved.
\end{proof}

Now, recall that for each $\alpha \in (0, 1]$ and each parabolic cylinder $Q \subset \bR^{d+1}$, the parabolic H\"{o}lder semi-norm of the function $u$ defined in $Q$ is
\[
[[u]]_{C^{\alpha/2, \alpha}(Q)} = \sup_{\substack{(t,x), (s,y) \in Q\\ (t,x) \not=(s,y)}} \frac{|u(t,x)-u(s,y)|}{|t-s|^{\alpha/2} + |x-y|^{\alpha}},
\]
and its H\"{o}lder norm is
\[
\|u\|_{C^{\alpha/2, \alpha}(Q)} = \|u\|_{L_\infty(Q)} + [[u]]_{C^{\alpha/2, \alpha}(Q)}.
\]
The following lemma is needed later in this paper.
\begin{lemma}
                    \label{lem1.3}
Under the assumptions of Lemma \ref{lem1.2}, we have
\begin{equation}
                                    \label{eq9.40}
\|\omega\|_{C^{1/2,1}(Q_{1/2}^+)}\le N(d,\nu,q_0)\|\omega\|_{L_{q_0}(Q_{1}^+)},
\end{equation}
and for any $\alpha\in (0,1)$,
\begin{equation*}
\|D\omega\|_{C^{\alpha/2,\alpha}(Q_{1/2}^+)}\le N(d,\nu,\alpha, q_0)\|D\omega\|_{L_{q_0}(Q_{1}^+)},
\end{equation*}
where $\omega = (\omega_{kl})_{k,l=1}^d$ is defined in \eqref{omega-def}.
\end{lemma}
\begin{proof} Let
$\tilde{u}_k(t,x)$ be the even extensions of $u_k(t,x)$ with respect to $x_d$, $k=1,\ldots,d-1$, and $\tilde{u}_d(t,x)$ be the odd extension of $u_d(t,x)$ with respect to $x_d$.
Further, let $\tilde{p}$ be the even extension of $p$ in $x_d$.
Set
$$
\bar{a}_{ij} = a_{ij}(t) \quad \text{for} \  i,j=1,\ldots,d-1,
\quad \bar{a}_{dd}=a_{dd}(t),
$$
$$
\bar{a}_{jd} = \bar{a}_{dj} = \left\{
\begin{aligned}
a_{jd}(t) &\quad x_d > 0,
\\
- a_{jd}(t) &\quad x_d < 0,
\end{aligned}
\right.
\quad \text{for} \  j=1,\ldots,d-1.
$$
Then by the boundary conditions on $u$, we have $\tilde{u} \in W_{q_0}^{1,2}(Q_1)^d$, $ \tilde{p} \in W_1^{0,1}(Q_1)$, and
$$
\tilde{u}_t - \bar{a}_{ij}(t,x_d) D_{ij} \tilde{u} + \nabla \tilde{p} = 0 \quad \text{in} \  Q_1.$$

We again denote by $\omega_{kl}$ the extensions of those $\omega_{kl}$ defined in the proof of Lemma \ref{lem1.2} with respect to $x_d$.
That is, $\omega_{kl}$, $k,l \in \{1,\ldots,d-1\}$, is even and $\omega_{dl}$, $l \in \{1,\ldots,d-1\}$,  is odd with respect to $x_d$, so
$$
\omega_{kl} = \partial_k \tilde{u}_l - \partial_l \tilde{u}_k
$$
in $Q_1$ for $k,l \in \{1,2,\ldots,d\}$.
It is easily seen that $\omega_{kl}$ satisfies the following equation in divergence form:
\begin{equation*}
\partial_t \omega_{kl} - \Div (\tilde{a}^{\top} \nabla \omega_{kl})   =  0
\end{equation*}
in $Q_1$, $k,l \in \{1,2,\ldots,d\}$,
where
\[
\begin{split}
\tilde{a}_{dj} & =\tilde{a}_{dj}(t, x_d)
= \left\{
\begin{array}{ll}
 2a_{dj}(t) & \quad x_d \geq 0,
\\
- 2a_{dj}(t) & \quad x_d <0,
\end{array}
\right.
\\
\tilde{a}_{jd} & = \tilde{a}_{jd}(t, x_d) = 0,
\end{split}
\]
for $j =1, 2,\ldots, d-1$,
and
$$
\tilde{a}_{ij} = a_{ij}(t), \quad \text{and} \quad \tilde{a}_{dd} = a_{dd}(t)
$$
for $i, j \in \{1, 2,\ldots, d-1\}$.

If we know a priori that $\omega_{kl}$ is sufficiently smooth, then $\omega_{kl}$ also  satisfies the non-divergence form equation
\begin{equation}
							\label{eq0712_01}
\partial_t \omega_{kl} - \bar{a}_{ij}(t,x_d) D_{ij} \omega_{kl} = 0
\end{equation}
in $Q_1$.
While checking this, we use the identity
\begin{equation}
							\label{eq0717_01}
D_d \left( \bar{a}_{dj}(t,x_d) D_{dj} \tilde{u}_l \right) = \bar{a}_{dj}(t,x_d) D_{dj} D_d \tilde{u}_l
\end{equation}
for $l = 1, \ldots, d-1$,
which follows from the definition of $\bar{a}_{dj}$ and the evenness of $\tilde{u}_l$ with respect to $x_d$.
Indeed, one can show that $\omega_{kl}$ belongs to $W_{q_0}^{1,2}(Q_r)$ for any $r \in (0,1)$ and  satisfies \eqref{eq0712_01} in $Q_r$ by using the $W^{1,2}_p$ solvability of parabolic equations in non-divergence form with coefficients being measurable functions of $(t,x_d)$ except for $a_{dd}$, which is a measurable function of $t$ only (cf. \cite{MR2550072,MR2833589}), as well as the unique solvability of the divergence form equation for $\omega_{kl}$ (cf. \cite{MR2764911}).

Once we check that $\omega_{kl} \in W_{q_0}^{1,2}(Q_r)$, $r \in (0,1)$, satisfies \eqref{eq0712_01}, we use the parabolic Sobolev embedding theorem combined with bootstrap and iterations to obtain \eqref{eq9.40}.

Since the coefficients in \eqref{eq0712_01} are independent of $x'$, by differentiating \eqref{eq0712_01} in $x'$ (in fact, using finite-difference quotients), we find that $D_{x'}\omega_{kl} \in W_{q_0}^{1,2}(Q_r)$, $r \in (0,1)$, also satisfies \eqref{eq0712_01}. This together with \eqref{eq9.40} shows that
\begin{equation*}
\|D_{x'}\omega\|_{C^{1/2,1}(Q_{1/2}^+)}\le N(d,\nu,q_0)\|D_{x'}\omega\|_{L_{q_0}(Q_{1}^+)}.
\end{equation*}

It remains to estimate $D_d \omega$.
For $k,l \in \{1,\ldots,d-1\}$, using the evenness of $\omega_{kl}$ and the unique solvability of the non-divergence and divergence form equations as above, we notice that $D_d\omega_{kl}$ belongs to $W_{q_0}^{1,2}(Q_r)$, $r \in (0,1)$, and satisfies
$$
\partial_t D_d \omega_{kl} - \bar{a}_{ij}(t,x_d) D_{ij} D_d \omega_{kl} = 0
$$
in $Q_r$.
Then, by the same reasoning as above, we obtain
$$
\|D_{d}\omega_{kl}\|_{C^{1/2,1}(Q_{1/2}^+)}\le N(d,\nu,q_0)\|D_{d}\omega_{kl}\|_{L_{q_0}(Q_{1}^+)}.
$$

Finally, for $l=1,\ldots,d-1$, $\omega_{dl}$ satisfies \eqref{eq0712_01} in $Q_1^+$ with the Dirichlet boundary condition on $Q_1 \cap \{(t,x',x_d) \in \bR^{d+1}: x_d = 0\}$.
Thus, by the boundary $W^{1,2}_p$ estimate with $p>(d+2)/(1-\alpha)$, the parabolic Sobolev embedding theorem, and the boundary Poincar\'e inequality, we have
$$
\|D_d\omega_{dl}\|_{C^{\alpha/2,\alpha}(Q_{1/2}^+)}
\le N\|\omega_{dl}\|_{L_{q_0}(Q_{1}^+)}
\le N\|D_d\omega_{dl}\|_{L_{q_0}(Q_{1}^+)},
$$
where $N=N(d,\nu,\alpha, q_0)$. The lemma is proved.
\end{proof}
\begin{remark}
Lemma \ref{lem1.3} can also be proved by using the boundary $W^{1,2}_q$-estimate with either the Dirichlet or Neumann boundary condition. We give a sketch below. Recall that $\omega_{kl}$ satisfies the divergence form equation \eqref{eq8.05} with either the conormal or Dirichlet boundary condition. Since the coefficients are independent of $x$, we can use the uniqueness of strong solutions in the half-space to show that $\omega_{kl}$ is also in $W^{1,2}_q(Q_{1/2}^+)$ for any $q<\infty$. See, for instance, \cite{MR2550072}. To obtain the estimates in Lemma \ref{lem1.3}, it remains to use the parabolic Sobolev embedding theorem.
\end{remark}

\section{A solvability result: proof of Theorem \ref{existence-lemma}}      \label{sec4}

In this section, we prove Theorem \ref{existence-lemma}, which demonstrates the existence of a solution to the system \eqref{u-sol.eqn-whole} with the boundary conditions \eqref{Lions-bdr-c}.
Henceforth, we denote
\[
\bR^d_T = (0, T) \times \bR^d.
\]
We first give a lemma.

\begin{lemma}
							\label{lem0715_1}
Let $T > 0$, $q_0 \in (1,\infty)$, $q_1 \in (1,d)$, $h \in L_{q_0}(\bR^d_T) \cap L_{q_1}(\bR^d_T)$, and $k \in \{1, \ldots, d\}$.
We define
$$
v_k(t,x) = \int_{\bR^d} D_k \Phi(x-y) h(t,y) \, dy
$$
in $\bR^d_T$, where $\Phi(\cdot)$ is the fundamental solution of the Laplace equation in $\bR^d$.
Then we have the following.
\begin{enumerate}
\item $\displaystyle \int_0^T \|v_k(t,\cdot)\|_{L_{q_1^*}(\bR^d)}^{q_1} \, dt < \infty$ and $D_x v_k \in L_{q_0}(\bR^d_T)$ with the estimates
\begin{equation}
							\label{eq0715_01}
\left( \int_0^T \|v_k(t,\cdot)\|_{L_{q_1^*}(\bR^d)}^{q_1} \, dt \right)^{\frac 1 {q_1}} \leq N(d,q_1) \|h\|_{L_{q_1}(\bR^d_T)},
\end{equation}
\begin{equation}
							\label{eq0715_08}
\|D_x v_k\|_{L_{q_0}(\bR^d_T)} \leq N(d,q_0) \|h\|_{L_{q_0}(\bR^d_T)},
\end{equation}
where $q_1^* = d q_1/(d-q_1)$.
We also have
\begin{equation}
							\label{eq0715_14}
\sum_{k=1}^d D_k v_k(t,x) = h(t,x) \quad \text{in} \  \bR^d_T.
\end{equation}

\item If $D_x h \in L_{q_0}(\bR^d_T)$, then
$D^2_x v_k \in L_{q_0}(\bR^d_T)$ with
\begin{equation}
							\label{eq0715_15}
\Delta v_k(t,x) = D_k h(t,x) \quad \text{in} \  \bR^d_T,
\end{equation}
and the following estimate holds:
\begin{equation}
							\label{eq0715_16}
\|D^2_x v_k\|_{L_{q_0}(\bR^d_T)} \leq N(d,q_0) \|D_k h\|_{L_{q_0}(\bR^d_T)}.
\end{equation}

\item If $D_x h \in L_{q_0}(\bR^d_T) \cap L_{q_1}(\bR^d_T)$, it holds that
\begin{equation}
							\label{eq0715_17}
D v_k(t,x) = \int_{\bR^d} D_k \Phi(x-y) D h(t,y) \, dy \quad \text{in} \  \bR^d_T.
\end{equation}

\end{enumerate}
\end{lemma}

\begin{proof}
For a.e. $t \in [0,T]$, $h(t,\cdot) \in L_{q_1}(\bR^d)$.
Thus, for a.e. $t \in [0,T]$, by the Hardy--Littlewood--Sobolev theorem of fractional integration (see \cite[Chapter V]{MR0290095}), we have
$$
\|v_k(t,\cdot)\|_{L_{q_1^*}(\bR^d)} \leq N(d,q_1)\|h(t,\cdot)\|_{L_{q_1}(\bR^d)}.
$$
By integrating both sides of the above inequality with respect to $t \in [0,T]$, we obtain \eqref{eq0715_01}.

Now we find $h^m(t,x) \in C_0^\infty\left([0,T] \times \bR^d\right)$ such that
\begin{equation}
							\label{eq0716_01}
\|h^m - h\|_{L_{q_1}(\bR^d_T)} + \|h^m - h\|_{L_{q_0}(\bR^d_T)} \to 0
\end{equation}
as $m \to \infty$.
For each $m = 1,2,\ldots$, we set
\begin{equation}
							\label{eq0715_04}
v_k^m(t,x) := \int_{\bR^d} D_k \Phi(x-y) h^m(t,y) \, dy.
\end{equation}
Then, again by the Hardy--Littlewood--Sobolev theorem,
we have
\begin{equation}
							\label{eq0715_05}
\int_0^T \|v_k^m(t,\cdot) - v_k(t,\cdot)\|_{L_{q_1^*}(\bR^d)}^{q_1} \, dt \to 0
\end{equation}
as $m \to \infty$.
Integration by parts applied to \eqref{eq0715_04} gives
\begin{equation}
							\label{eq0715_02}
v_k^m(t,x) = \int_{\bR^d} \Phi(x-y) D_k h^m(t,y) \, dy,
\end{equation}
from which it follows that
\begin{equation}
							\label{eq0715_06}
\Delta v_k^m (t,x) = D_k h^m(t,y).
\end{equation}
Note that
\begin{equation}
							\label{eq0715_03}
v_k^m(t,x) = D_k \int_{\bR^d} \Phi(x-y) h^m(t,y) \, dy
\end{equation}
and
\begin{equation}
							\label{eq0715_18}
D v_k^m(t,x) = \int_{\bR^d} D_k \Phi(x-y) D h^m(t,y) \, dy = DD_k \int_{\bR^d} \Phi(x-y) h^m(t,y) \, dy.
\end{equation}
Thus,
\begin{equation}
							\label{eq0715_13}
\sum_{k=1}^d D_k v_k^m(t,x) = \Delta \int_{\bR^d} \Phi(x-y) h^m(t,y) \, dy = h^m(t,x)
\end{equation}
in $\bR^d_T$.
By applying to \eqref{eq0715_03} the fact that the double Riesz transform is bounded in $L_p(\bR^d)$, $1<p<\infty$, (see \cite[Chapter III]{MR0290095}) and by integrating both sides of the obtained inequality in $t$, we arrive at
$$
\|D_x v_k^m\|_{L_{q_0}(\bR^d_T)} \leq N(d,q_0) \|h^m\|_{L_{q_0}(\bR^d_T)}.
$$
Then, \eqref{eq0715_08} and \eqref{eq0715_14} follow from this inequality, \eqref{eq0715_13}, \eqref{eq0716_01}, and \eqref{eq0715_05}.

If $D_x h \in L_{q_0}(\bR^d_T)$, we find $h^m \in C_0^\infty\left([0,T] \times \bR^d\right)$ such that
\begin{equation}
							\label{eq0715_19}
\|h^m - h\|_{L_{q_1}(\bR^d_T)} + \|h^m - h\|_{L_{q_0}(\bR^d_T)} + \|D_x h^m - D_x h\|_{L_{q_0}(\bR^d_T)} \to 0
\end{equation}
as $m \to \infty$.
Then, by applying the boundedness in $L_p(\bR^d)$ of the double Riesz transform to \eqref{eq0715_02}, we obtain
$$
\|D_x^2 v^m\|_{L_{q_0}(\bR^d_T)} \leq N(d,q_0) \|D_k h^m\|_{L_{q_0}(\bR^d_T)}.
$$
Using this estimate, \eqref{eq0715_05}, \eqref{eq0715_06}, and \eqref{eq0715_19}, we prove \eqref{eq0715_15} and \eqref{eq0715_16}.

Finally, to prove \eqref{eq0715_17}, we use the Hardy--Littlewood--Sobolev theorem as well as the first equality in \eqref{eq0715_18} with $h^m \in C_0^\infty\left([0,T] \times \bR^d\right)$ satisfying \eqref{eq0715_19} as well as
$$
\|D_x h^m - D_x h\|_{L_{q_1}(\bR^d_T)} \to 0
$$
as $m \to \infty$.
The lemma is proved.
\end{proof}

\begin{proposition}
							\label{prop0718_1}
Let $T >0$, $q_0 \in (1, \infty)$.
Assume that $a_{ij} = a_{ij}(t)$ for $t \in (0,T)$.
Let $f \in L_{q_0}((0, T) \times \bR_{+}^d)^d$ and $g\in L_{q_0}((0, T) \times \bR_{+}^d)$, $Dg \in L_{q_0}((0, T) \times \bR_{+}^d)^d$, $g(0,\cdot) = 0$, and $g_t = \Div(G)$ in the sense of \eqref{eq0702_01} for some vector field
$$
G = (G_1, G_2,\ldots, G_d) \in L_{q_0}((0, T) \times \bR_{+}^d)^d.
$$
Additionally, assume that $g$ and $G$ vanish for large $|x|$ uniformly in $t \in [0,T]$ and $f \in C_0^\infty((0,T) \times \bR^d_{+})^d$.
Then, there exists a solution $(u, p)$ of \eqref{u-sol.eqn-whole}--\eqref{Lions-bdr-c} such that
\[
\begin{split}
& u \in L_\infty\left((0,T), L_{q_0}(\bR^d_+)\right), \quad u_t, Du, D^2 u \in L_{q_0}((0, T) \times \bR_{+}^d), \\
& p \in L_{q_0}((0, T), L_{q_0, \textup{loc}}(\bR^d_+)), \quad \nabla p  \in L_{q_0}((0, T) \times \bR_{+}^d )
\end{split}
\]
and that satisfies \eqref{eq0715_07} and \eqref{eq0715_09}.
\end{proposition}

\begin{proof}
Set $\bar{a}_{ij}$ to be as defined in the proof of Lemma \ref{lem1.3}.
Our goal is to construct a strong solution $(\tilde{u}, \tilde{p})$ in $\bR^d_T = (0,T) \times \bR^d$ of the Stokes system
\begin{equation} \label{extended-solution-St}
\left\{
\begin{array}{cccl}
\tilde{u}_t -\bar{a}_{ij} D_{ij} \tilde{u} + \nabla \tilde{p} & = & \tilde{f}  & \quad \text{in} \ \bR_{T}^d,\\
\Div(\tilde{u}) & = & \tilde{g} & \quad \text{in} \  \bR_{T}^d,\\
 \tilde{u}(0,x) & =& 0& \quad \text{for} \ x \in \bR^d,
\end{array} \right.
\end{equation}
where $\tilde{g}(t,x)$ and $\tilde{f}(t,x) = (\tilde{f}_1, \tilde{f}_2,\ldots, \tilde{f}_d)$ are functions such that $\tilde{g}(t, \cdot)$ is the even extension of $g(t,\cdot)$ on $\bR^d$, $\tilde{f}_{k}(t,\cdot)$  is the even extension of $f_k(t, \cdot)$ on $\bR^d$, $k =1, 2,\ldots, d-1$, and $\tilde{f}_d(t,\cdot)$ is the odd extension of $f_{d}(t, \cdot)$.
We will see that the constructed solution $\tilde{u}_k(t,\cdot)$ is even in $x_d$ for all $k =1, 2,\ldots, d-1$, and $u_d(t,\cdot)$ is odd in the $x_d$-variable. Therefore, $u (t,\cdot) = \tilde{u}(t,\cdot)|_{ \bR^d_+}$ and $p(t,\cdot) = \tilde{p}(t,\cdot)|_{ \bR^d_+}$ satisfy \eqref{u-sol.eqn-whole}.

Since $g(t,x) = f(t,x) = G(t,x) =0$ for $|x|>R$ with a sufficiently large $R > 0$, there exists $q_1 \in (1,d)$ such that
$$
g, D_lg, f_l, G_l \in L_{q_0}((0, T) \times \bR^d_{+}) \cap L_{q_1}((0, T) \times \bR^d_{+}), \quad l = 1, \ldots, d.
$$

\noindent
{\bf Step 1}: We construct $\tilde{u}$ and prove the first two estimates in \eqref{eq0715_07}.
Recall that $f = (f_1,\ldots,f_d)$ is assumed to be smooth with compact support in $(0,T) \times \bR^d_+$.
Thus,
$$
D_k\tilde{f}_l - D_l \tilde{f}_k \in L_{q_0}(\bR^d_T),
$$
and according to the results in \cite{MR2833589}, there exist $\omega_{kl} \in W_{q_0}^{1,2}(\bR^d_T)$, $k,l \in \{1,\ldots,d\}$, satisfying the non-divergence form equations
\begin{equation}
							\label{eq0704_01}
\left\{
\begin{aligned}
\partial_t \omega_{kl} - \bar{a}_{ij} D_{ij}\omega_{kl} &= D_k \tilde{f}_l - D_l \tilde{f}_k \quad \text{in} \  \bR^d_T,
\\
\omega_{kl}(0,x) &= 0 \quad \text{for} \ x \in \bR^d.
\end{aligned}
\right.
\end{equation}
Since $\tilde{f}_k(t,x)$, $k=1,\ldots,d-1$,  is the even extension of $f_k(t,x)$ and $\tilde{f}_d(t,x)$ is the odd extension of $f_d(t,x)$,
$$
D_k\tilde{f}_l(t,x) - D_l \tilde{f}_k(t,x), \quad k,l \in \{1,\ldots,d-1\},
$$ is even with respect to $x_d$, and
$$
D_d \tilde{f}_l - D_l \tilde{f}_d, \quad l \in \{1,\ldots,d-1\},
$$ is odd with respect to $x_d$.
By the evenness and oddness of the right-hand sides and coefficients of \eqref{eq0704_01}, we see that $\omega_{kl}$, $k,l \in \{1,\ldots,d-1\}$,  is even with respect to $x_d$, and $\omega_{dl}$, $l \in \{1,\ldots,d-1\}$,  is odd with respect to $x_d$.
We also see that $\omega_{kl} = - \omega_{kl}$.
Furthermore, one can check that $\omega_{kl}$, $k,l \in \{1,\ldots,d\}$, satisfies the following divergence form equation:
\begin{equation} \label{eq0625_01}
\left\{
\begin{array}{cccl}
\partial_t \omega_{kl} - \Div (\tilde{a}^{\top} \nabla \omega_{kl})  & = & D_{k} \tilde{f}_l - D_l \tilde{f}_k  & \quad \text{in} \  \bR_{T}^d,
\\
\omega_{kl}(0,x) & =& 0& \quad \text{for} \ x \in \bR^d,
\end{array} \right.
\end{equation}
where $\tilde{a}_{ij}$ (which  is different from $\bar{a}_{ij}$)  is as defined in the proof of Lemma \ref{lem1.3}.
By the $\cH^1_p$ estimates (see \cite[Theorem 2.1]{MR2764911}), we have
\begin{equation}
							\label{eq0626_03}
\norm{D{\omega}_{kl}}_{L_{q_0}(\bR^d_T)}\le N(\nu, q_0) \norm{\tilde{f}}_{L_{q_0}(\bR^d_T)},\quad  \norm{{\omega}_{kl}}_{L_{q_0}(\bR^d_T)} \leq N(\nu, T, q_0) \norm{\tilde{f}}_{L_{q_0}(\bR^d_T)}.
\end{equation}
Since $f \in L_{q_1}((0,T) \times\bR_{+}^d)^d$, we also have
\begin{equation}
							\label{eq0626_04}
\norm{D{\omega}_{kl}}_{L_{q_1}(\bR^d_T)}\le N(\nu, q_1) \norm{\tilde{f}}_{L_{q_1}(\bR^d_T)},\quad  \norm{{\omega}_{kl}}_{L_{q_1}(\bR^d_T)} \leq N(\nu, T, q_1) \norm{\tilde{f}}_{L_{q_1}(\bR^d_T)}.
\end{equation}

Now we set
\begin{equation}
							\label{eq0715_10}
\tilde{u}_l(t,x) = \int_{\bR^d}D_l \Phi(x-y) \tilde{g}(t,y) \, dy + \sum_{k=1}^d \int_{\bR^d} D_k \Phi(x-y) \omega_{kl}(t,y) \, dy
\end{equation}
in $\bR^d_T$.
By the properties of the fundamental solution $\Phi(\cdot)$, $\tilde{g}$, and $\omega_{kl}$, we see that $\tilde{u}_l(t,\cdot)$, $l=1,\ldots,d-1$,  is even in $x_d$, and $\tilde{u}_d(t,\cdot)$ is odd in $x_d$.
Note that
$$
\tilde{g}, \, D_x \tilde{g}, \, \omega_{kl}, \, D_x \omega_{kl} \in L_{q_0}(\bR^d_T) \cap L_{q_1}(\bR^d_T).
$$
Then, by Lemma \ref{lem0715_1},
\begin{equation}
                                    \label{eq10.53}
\Delta \tilde{u}_l = D_l \tilde{g} + \sum_{k=1}^d D_k \omega_{kl}\quad\text{in}\,\,\bR^d_T,
\end{equation}
\begin{equation*}
\|D\tilde{u}_l\|_{L_{q_0}(\bR^d_T)} \leq N \|\tilde{g}\|_{L_{q_0}(\bR^d_T)} + N\|\omega_{kl}\|_{L_{q_0}(\bR^d_T)},
\end{equation*}
and
\begin{equation*}
\|D^2\tilde{u}_l\|_{L_{q_0}(\bR^d_T)} \leq N \|D\tilde{g}\|_{L_{q_0}(\bR^d_T)} + N \|D\omega_{kl}\|_{L_{q_0}(\bR^d_T)},
\end{equation*}
where $N = N(d,q_0)$.
These estimates combined with \eqref{eq0626_03} prove the first two estimates in \eqref{eq0715_07}, provided that $u(t,\cdot) = \tilde{u}(t,\cdot)|_{\bR^d_+}$ satisfies \eqref{u-sol.eqn-whole}.

\vspace{1em}

\noindent
{\bf Step 2}: We prove \eqref{eq0715_09}.
Observe that, for $\varphi \in C_0^\infty(\bR^d_T)$,
\begin{align*}
\int_{\bR^d_T} \tilde{u}_l \varphi_t \, dx \, dt &= \int_{\bR^d_T} \varphi_t(t,x) \int_{\bR^d} D_l \Phi(x-y) \tilde{g}(t,y) \, dy \, dx \, dt\\
&\quad + \sum_{k=1}^d\int_{\bR^d_T} \varphi_t(t,x) \int_{\bR^d} D_k \Phi(x-y) \omega_{kl}(t,y) \, dy \, dx \, dt\\
&= \int_{\bR^d} D_l\Phi(y) \int_{\bR^d_T} \tilde{g}(t,x) \varphi_t(t,x+y) \, dx \, dt \, dy\\
&\quad + \sum_{k=1}^d \int_{\bR^d} D_k \Phi(y) \int_{\bR^d_T} \omega_{kl}(t,x) \varphi_t(t,x+y) \, dx \, dt \, dy
=: J_1 + J_2.
\end{align*}
For $k=1,2,\ldots,d-1$, we set $\tilde G_k(t,\cdot)$ to be the even extension of $G_k(t,\cdot)$ with respect to $x_d$ and $\tilde G_d(t,\cdot)$ to be the odd extension of $G_d(t,\cdot)$ with respect to $x_d$.
By \eqref{eq0702_01} we have
\begin{equation*}
\partial_t \tilde{g} = \Div \tilde{G}
\end{equation*}
in the sense of distribution.
Then
\begin{align*}
J_1 &= \int_{\bR^d} D_l \Phi(y) \int_{\bR^d_T} \tilde{G}(t,x) \cdot \nabla  \varphi(t,x+y) \, dx \, dt \, dy\\
&= \sum_{k=1}^d \int_{\bR^d_T} D_k \varphi(t,x) \int_{\bR^d} D_l \Phi(x-y) \tilde{G}_k(t,y) \, dy \, dx \, dt.
\end{align*}
Since $\tilde{G}_k \in L_{q_0}(\bR^d_T) \cap L_{q_1}(\bR^d_T)$, by Lemma \ref{lem0715_1},
$$
V_{lk}(t,x) := \int_{\bR^d} D_l \Phi(x-y) \tilde{G}_k(t,y) \, dy
$$
satisfies $D_x V_{lk} \in L_{q_0}(\bR^d_T)$ with the estimate as in \eqref{eq0715_08}.
Thus, if we set $V_1(t,x) = \sum_{k=1}^d D_k V_{kl}(t,x)$, then
\begin{equation*}
J_1 = - \int_{\bR^d_T} V_1(t,x) \varphi(t,x) \, dx \, dt
\end{equation*}
and
$$
\|V_1\|_{L_{q_0}(\bR^d_T)} \leq N(d,q_0)\|\tilde{G}\|_{L_{q_0}(\bR^d_T)}.
$$
For $J_2$, we observe that from \eqref{eq0625_01},
\begin{align*}
&\int_{\bR^d_T} \omega_{kl}(t,x) \varphi_t(t,x+y) \, dx \, dt = \int_{\bR^d_T} \tilde{a}_{ji} D_j \omega_{kl}(t,x) D_i \varphi(t,x+y) \, dx \, dt\\
&\quad + \int_{\bR^d_T} \tilde{f}_l (t,x) D_k \varphi(t,x+y) \, dx \, dt - \int_{\bR^d_T} \tilde{f}_k(t,x) D_l\varphi(t,x+y) \, dx \, dt.
\end{align*}
From \eqref{eq0626_03} and \eqref{eq0626_04}, we see that $\tilde{a}_{ji} D_j \omega_{kl} \in L_{q_0}(\bR^d_T) \cap L_{q_1}(\bR^d_T)$ and $\tilde{f} \in  L_{q_0}(\bR^d_T) \cap L_{q_1}(\bR^d_T)$; thus,
by proceeding as above, we find that there exists $V_2 \in L_{q_0}(\bR^d)$ such that
$$
J_2 = - \int_{\bR^d_T} V_2(t,x) \varphi(t,x) \, dx \, dt
$$
and
$$
\|V_2\|_{L_{q_0}(\bR^d_T)} \leq N \|D\omega_{kl}\|_{L_{q_0}(\bR^d_T)} + N \|\tilde{f}\|_{L_{q_0}(\bR^d_T)} \leq N \|\tilde{f}\|_{L_{q_0}(\bR^d_T)},
$$
where $N = N(d, \nu, q_0)$, and the last inequality is due to  the first estimate in \eqref{eq0626_03}.
From the above observations on $J_1$ and $J_2$, we see that
$$
\partial_t\tilde{u}_l = V_1 + V_2,
$$
and
$$
\|\partial_t\tilde{u}_l\|_{L_{q_0}(\bR^d_T)} \leq N \|\tilde{G}\|_{L_{q_0}(\bR^d_T)} + N \|\tilde{f}\|_{L_{q_0}(\bR^d_T)},
$$
where $N = N(d,q_0,\nu)$.
This proves \eqref{eq0715_09}.

\vspace{1em}

\noindent
{\bf Step 3}: We prove that in $\bR^d_T$,
\begin{equation}
							\label{eq0715_12}
\Div \tilde{u}(t,x) = \tilde{g}(t,x)
\end{equation}
and
\begin{equation}
							\label{eq0716_02}
\partial_k \tilde{u}_l - \partial_l \tilde{u}_k = \omega_{kl}.
\end{equation}
By \eqref{eq0715_10} and Lemma \ref{lem0715_1}, one can write
\begin{equation}
							\label{eq0715_11}
\sum_{l=1}^d D_l \tilde{u}_l = \sum_{l=1}^d D_l \int_{\bR^d} D_l \Phi(x-y) \tilde{g}(t,y) \, dy + \sum_{l,k=1}^d D_l \int_{\bR^d} D_k \Phi(x-y) \omega_{kl}(t,y) \, dy,
\end{equation}
where the second term is zero because $\omega_{kl} = - \omega_{lk}$.
Regarding the first term in \eqref{eq0715_11}, we observe that
\begin{align*}
&\sum_{l=1}^d\int_{\bR^d_T} \tilde{u}_l D_l \varphi \, dx \, dt = \sum_{l=1}^d \int_{\bR^d_T} D_l \varphi(t,x) \int_{\bR^d} D_l \Phi(x-y) \tilde{g}(t,y) \, dy \, dx \, dt\\
&= \sum_{l=1}^d \int_{\bR^d_T} \tilde{g}(t,x) \int_{\bR^d} (D_l \Phi)(y-x) D_l \varphi(t,y) \, dy \, dx \, dt\\
&= - \int_{\bR^d_T} \tilde{g}(t,x) \Delta \int_{\bR^d} \Phi(x-y) \varphi(t,y) \, dy \, dx \, dt = -  \int_{\bR^d_T} \tilde{g}(t,x) \varphi(t,x) \, dx \, dt
\end{align*}
for any $\varphi \in C_0^\infty(\bR^d_T)$.
Hence, \eqref{eq0715_12} is proved.

To prove \eqref{eq0716_02}, we first show that
\begin{equation}
							\label{eq0715_20}
\partial_k \omega_{jl} - \partial_l \omega_{jk} = \partial_j \omega_{kl}
\end{equation}
in $\bR^d_T$ for all $k, j, l \in \{1,\ldots,d\}$.
By the properties of $\omega_{kl}$, this is equivalent to showing that
\begin{equation}
							\label{eq0704_02}
\partial_k \omega_{jl} + \partial_j\omega_{lk} + \partial_l \omega_{kj} = 0
\end{equation}
in $\bR^d_T$.
It is sufficient to check \eqref{eq0704_02}
for three cases:
$k,j,l \in \{1,\ldots,d-1\}$, $k = d$, $j,l \in \{1,\ldots,d-1\}$, and
 $k=j=d$ and $l\in \{1,\ldots,d-1\}$.
In the last case, \eqref{eq0704_02} becomes
$$
\partial_d \omega_{dl} + \partial_d \omega_{ld} = 0,
$$
which is guaranteed by the property of $\omega_{kl}$.
For the first and second cases, by differentiating the equations \eqref{eq0704_01} in $x_r$, we write them as
\begin{equation}
							\label{eq0715_21}
\left\{
\begin{aligned}
\partial_t D_r\omega_{kl} - \bar{a}_{ij} D_{ij}D_r \omega_{kl} &= D_r (D_k \tilde{f}_l - D_l \tilde{f}_k) \quad \text{in} \ \bR^d_T,
\\
D_r\omega_{kl}(0,x) &= 0 \quad \text{for} \ x \in \bR^d,
\end{aligned}
\right.
\end{equation}
where
$$
\left\{
\begin{aligned}
r&=1,\ldots,d \quad \text{if} \quad 	k,l\in \{1,\ldots,d-1\},
\\
r&=1,\ldots,d-1 \quad \text{if} \quad k = d \,\, \text{or} \,\, l = d.
\end{aligned}
\right.
$$
In particular, note that when $k,l \in \{1,\ldots,d-1\}$,
$$
D_d \bar{a}_{ij} D_{ij} \omega_{kl} = \bar{a}_{ij} D_{ij}D_d \omega_{kl},
$$
which follows from the evenness of $\omega_{kl}$ with respect to $x_d$.
From \eqref{eq0715_21}, one can see that $\partial_k \omega_{jl} + \partial_j\omega_{lk} + \partial_l\omega_{kj} \in W_{q_0}^{1,2}(\bR^d_T)$ satisfies \eqref{eq0715_21} with the right-hand side being zero.
Then, by uniqueness, \eqref{eq0704_02} follows.

Now we prove \eqref{eq0716_02}.
From \eqref{eq0715_10} we have
\begin{align*}
\partial_k \tilde{u}_l - \partial_l \tilde{u}_k &= D_k \int_{\bR^d} D_l \Phi(x-y) \tilde{g}(t,y) \, dy - D_l \int_{\bR^d} D_k \Phi(x-y) \tilde{g}(t,y) \, dy\\
&\quad + \sum_{r=1}^d \left[D_k \int_{\bR^d} D_r \Phi(x-y) \omega_{rl}(t,y) \, dy - D_l \int_{\bR^d} D_r \Phi(x-y) \omega_{rk}(t,y) \, dy \right],
\end{align*}
where the first two terms on the right-hand side cancel each other.
Then, since $D\omega_{rl} \in L_{q_0}(\bR^d_T) \cap L_{q_1}(\bR^d_T)$, by \eqref{eq0715_17}
\begin{align*}
\partial_k \tilde{u}_k - \partial_k \tilde{u}_k &= \sum_{r=1}^d \int_{\bR^d} D_r \Phi(x-y) \left( D_k \omega_{rl}(t,y) - D_l \omega_{rk}(t,y) \right) \, dy\\
&= \sum_{r=1}^d \int_{\bR^d} D_r \Phi(x-y) D_r \omega_{kl}(t,y) \, dy = \omega_{kl},
\end{align*}
where we used \eqref{eq0715_20} in the second equality and \eqref{eq0715_17} as well as \eqref{eq0715_14} in the last equality.

\vspace{1em}

\noindent
{\bf Step 4}:
We prove that there exists $\tilde{p}: \bR^d_T \to \bR$ such that $(\tilde{u}, \tilde{p})$ satisfies \eqref{extended-solution-St}.
In fact, the second relation in \eqref{extended-solution-St} is shown in \eqref{eq0715_12}, and the third one follows from the definition of $\tilde{u}$ in \eqref{eq0715_10} and the initial conditions on $g$ and $\omega_{kl}$.
Thus, we prove here
\begin{equation}
							\label{eq0716_04}
\tilde{u}_t -\bar{a}_{ij} D_{ij} \tilde{u} + \nabla \tilde{p} = \tilde{f}
\end{equation}
in $\bR_{T}^d$.
Once this is proved, the last estimate in \eqref{eq0715_07} follows from \eqref{eq0716_04}, the second estimate in \eqref{eq0715_07}, and \eqref{eq0715_09}, as well as the evenness and oddness of the involved functions.
We set $h = (h_1,\ldots, h_d)$, where
$$
h_l(t,x) = \tilde{f}_l(t,x) - \partial_t \tilde{u}_l(t,x) + \bar{a}_{ij}(t,x_d) D_{ij}\tilde{u}_l(t,x)
$$
in $\bR^d_T$.
Then, using \eqref{eq0716_02} and \eqref{eq0625_01}, we see that
$$
D_k h_l - D_l h_k = 0
$$
in $\bR^d_T$ in the distribution sense.
In particular, if $k=d$ and $l \in \{1,\ldots,d-1\}$, then it follows that
\begin{align*}
D_d h_l - D_l h_d &= D_d \tilde{f}_l - D_l \tilde{f}_d - \partial_t (D_d \tilde{u}_l - D_l \tilde{u}_d) + D_d \left( \bar{a}_{ij} D_{ij} \tilde{u}_l \right) - D_l \left( \bar{a}_{ij} D_{ij} \tilde{u}_d \right)\\
&= D_d \tilde{f}_l - D_l \tilde{f}_d - \partial_t \omega_{dl} + D_i \left( \tilde{a}_{ji} D_j \omega_{dl} \right) = 0,
\end{align*}
where we used \eqref{eq0717_01}, which was deduced from the evenness of $\tilde{u}_l$ in $x_d$.

We extend $h_l$ to be zero for $t < 0$ and take infinitely differentiable functions $\eta(t) \in C_0^\infty(\bR)$ and $\zeta(x) \in C_0^\infty(\bR^d)$ with unit integrals  such that
$\eta(t) = 0$ for $t \geq 0$.
We set
$$
h_l^{(\varepsilon)}(t,x) = \int_{-\infty}^T\int_{\bR^d} h_l(s,y) \phi_\varepsilon(t-s,x-y) \, dy \, ds,
$$
where
$$
\phi_\varepsilon(t,x) = \varepsilon^{-d-2}\eta(t/\varepsilon^2)\zeta(x/\varepsilon).
$$
Then $h_l^{(\varepsilon)} \in C^\infty\left([0,T] \times \bR^d\right)$ and
\begin{equation}
							\label{eq0716_03}
D_k h_l^{(\varepsilon)} - D_l h_k^{(\varepsilon)} = 0
\end{equation}
in $\bR^d_T$.
Let
\begin{align*}
p^\varepsilon(t,x) &= \int_0^{x_1} h_1^{(\varepsilon)}
(t,r,0,\ldots,0) \, dr + \int_0^{x_2} h_2^{(\varepsilon)}(t,x_1,r,0,\ldots,0) \, dr+ \ldots\\
&\quad + \int_0^{x_d} h_d^{(\varepsilon)}(t,x_1,x_2,\ldots, x_{d-1},r) \, dr.
\end{align*}
We define
$$
\tilde{p}^\varepsilon(t,x) = p^\varepsilon(t,x) - \frac{1}{B_1} \int_{B_1} p^\varepsilon(t,y) \, dy.
$$
Using \eqref{eq0716_03}, we see that
$$
\nabla \tilde{p}^\varepsilon = \left(h_1^{(\varepsilon)}, \ldots, h_d^{(\varepsilon)}\right)
$$
in $\bR^d_T$, and
$$
\|\nabla \tilde{p}^\varepsilon\|_{L_{q_0}(\bR^d_T)} =
\|h^{(\varepsilon)}\|_{L_{q_0}(\bR^d_T)} \leq \|h\|_{L_{q_0}(\bR^d_T)}
$$
is bounded uniformly in $\varepsilon > 0$.

On the other hand, for each $R > 1$, by the Poincar\'{e} inequality,
$$
\|\tilde{p}^\varepsilon(t,\cdot)\|_{L_{q_0}(B_R)} \leq N(d,q_0,R) \|\nabla \tilde{p}^\varepsilon(t,\cdot)\|_{L_{q_0}(B_R)}
$$
for each $t \in [0,T]$.
By integrating both sides of the above inequality in $t$, we obtain
$$
\|\tilde{p}^\varepsilon\|_{L_{q_0}((0,T) \times B_R)} \leq N \|\nabla \tilde{p}^\varepsilon\|_{L_{q_0}((0,T) \times B_R)},
$$
which is bounded uniformly in $\varepsilon > 0$.
Hence, there exists $\tilde{p}(t,x)$ defined in $\bR^d_T$, which is spatially locally in $L_{q_0}(\bR^d_T)$, such that $\nabla\tilde{p} \in L_{q_0}(\bR^d_T)$ and a subsequence of $\{\nabla \tilde{p}^\varepsilon\}$ converges weakly to $\nabla\tilde{p}$ in $L_{q_0}(\bR^d_T)$.
Since
$$
D_l \tilde{p}^\varepsilon = h_l^{(\varepsilon)}, \quad l = 1,\ldots,d,
$$
in $\bR^d_T$ and $h_l^{(\varepsilon)} \to h_l$ in $L_{q_0}(\bR^d_T)$, we conclude that $(\tilde{u}, \tilde{p})$ satisfies \eqref{eq0716_04}.

Finally, we note from \eqref{eq0716_04} that $\tilde{u} \in L_\infty\left((0,T), L_{q_0}(\bR^d)\right)$.
\end{proof}

\begin{proof}[Proof of Theorem \ref{existence-lemma}]
We denote  $\bR^d_{T, +} = (0,T) \times \bR^d_+$.
Thanks to Proposition \ref{prop0718_1} and the fact that $C_0^\infty(\bR^d_{T,+})$ is dense in $L_{q_0}(\bR^d_{T,+})$, we need to show only that there exist functions $g^m$ and $G^m = (G_1^m,\ldots,G_d^m)$ defined on $\bR^d_{T,+}$ such that $g^m(t,x)$ and $G^m(t,x)$ vanish for large $|x|$ uniformly in $t \in [0,T]$,
$$
g^m, \, |D g^m| \in L_{q_0}(\bR^d_{T,+}), \quad g^m(0,\cdot) = 0, \quad G^m \in L_{q_0}(\bR^d_{T,+})^d,
$$
$$
\partial_t g^m = \Div G^m
$$
in $\bR^d_{T,+}$, and
$$
\|g - g^m\|_{L_{q_0}(\bR^d_{T,+})} + \|D g - Dg^m\|_{L_{q_0}(\bR^d_{T,+})} + \|G - G^m\|_{L_{q_0}(\bR^d_{T,+})} \to 0\quad\text{as}\ m \to \infty.
$$

We take infinitely differentiable functions
$\chi_m(x)$ defined on $\bR^d$ such that
$\chi_m(x) = 1$ on $B_{m/2}$ and $\chi_m(x) = 0$ for $x \in \bR^d \setminus B_m$, $m = 1, 2,\ldots$, where
$$
B_m := \{|x| < m: x \in \bR^d\}.
$$
We set
$$
B_m^+ = B_m \cap \{x_d > 0\}, \quad
c_m(t) = \frac{\int_{B_m^+} \nabla \chi_m(y) \cdot
G(t,y) \, dy}{\int_{B_m^+} \chi_m(y) \, dy},
$$
and find $H^m$ in $(0,T) \times B_m^+$ such that
$$
\left\{
\begin{aligned}
\Div H^m &= - \nabla \chi_m \cdot G + c_m(t) \chi_m(x) \quad \text{in} \  (0,T) \times B_m^+,
\\
H^m(t,x) &= 0 \quad \text{on} \ (t,x) \in (0,T) \times \partial B_m^+,
\end{aligned}
\right.
$$
with the estimate
$$
\|D_x H^m\|_{L_{q_0}\left((0,T) \times B_m^+\right)} \leq N(d,q_0) \left(\|\nabla \chi_m \cdot G\|_{L_{q_0}\left((0,T) \times B_m^+ \right)} + \|c_m(t)\chi_m(x)\|_{L_{q_0}\left((0,T) \times B_m^+ \right)} \right).
$$
This is indeed possible by using an integral representation of the solutions to the divergence equations on star-shaped domains, as shown in, for instance, \cite{MR1880723}.
We note that
$$
\|\nabla \chi_m \cdot G\|_{L_{q_0}\left((0,T) \times B_m^+ \right)} + \|c_m(t)\chi_m(x)\|_{L_{q_0}\left((0,T) \times B_m^+ \right)} \leq N(d) m^{-1} \| 1_{B_m \setminus B_{m/2}} G\|_{L_{q_0}(\bR^d_{T,+})}.
$$
From the above two inequalities, the Poincar\'{e} inequality on $B_m^+$, and the fact that
$$
\| 1_{B_m \setminus B_{m/2}} G\|_{L_{q_0}(\bR^d_{T,+})} \to 0\quad\text{as}\ m \to \infty,
$$
we have
$$
\|H^m\|_{L_{q_0}\left((0,T) \times B_m^+\right)} \leq m \|D_x H^m\|_{L_{q_0}\left((0,T) \times B_m^+\right)} \to 0\quad\text{as}\ m \to \infty.
$$

We set
$$
g^m(t,x) := \chi_m(x) g(t,x) + \chi_m(x) \int_0^t c_m(s) \, ds
$$
and
$$
G^m(t,x) :=
\left\{
\begin{aligned}
\chi_m(x) G(t,x) + H^m(t,x) \quad &\text{in} \  (t,x) \in (0,T) \times B_m^+,
\\
0 \quad &\text{in} \ (t,x) \in (0,T) \times (\bR^d_+ \setminus B_m^+).
\end{aligned}
\right.
$$
We then see that $g^m$ and $G^m$ satisfy the required properties.
In particular, for $\varphi \in C_0^\infty (\bR^d_T)$,
\begin{align*}
\int_{\bR^d_{T,+}} g^m \varphi_t \, dx \, dt &= \int_{\bR^d_{T,+}} \left[ \chi_m G \cdot \nabla \varphi +  \varphi \nabla \chi_m \cdot G + \chi_m \left(\int_0^t c_m(s) \, ds \right) \varphi_t  \right] \, dx \, dt\\
&= \int_{\bR^d_{T,+}} \left( \chi_m G \cdot \nabla \varphi +  \varphi \nabla \chi_m \cdot G - \chi_m(x) c_m(t) \varphi \right) \, dx \, dt\\
&= \int_{\bR^d_{T,+}} \chi_m G \cdot \nabla \varphi \, dx \, dt - \int_0^T \int_{B_m^+}  \varphi \Div H^m\, dx \, dt\\
&= \int_{\bR^d_{T,+}} G^m \cdot \nabla \varphi \, dx \, dt,
\end{align*}
where we used the fact that $H^m = 0$ on $(0,T) \times \partial B_m^+$.
The theorem is proved.
\end{proof}

\section{Stokes system with measurable coefficients and proof of Theorem \ref{thm2.3b}.} \label{non-div-se}

In this section, we consider the non-divergence form Stokes system with measurable coefficients and the Lions boundary conditions.
We give the proof of the main result of the paper, Theorem \ref{thm2.3b}.
Throughout this section, for any locally integrable function $f$ defined in a neighborhood of the parabolic cylinder $Q = \Gamma \times \Omega \subset \bR \times \bR^{d}$, we denote
\[
(f)_Q =\fint_{Q} f(t,x)\ dx \ dt, \quad \text{and} \quad [f]_{\Omega}(t) =\fint_{\Omega} f(t,x)\ dx, \quad t \in \Gamma.
\]
For a domain $\Omega\subset \bR^d_+$ and $\rho>0$, we denote
$$
\Omega^\rho=  {\textstyle \bigcup_{y \in \Omega}} B_{\rho}^+(y).
$$
We say that $\Omega$ satisfies the interior measure condition if there exists $\gamma\in (0,1)$ such that for any $x_0\in \overline{\Omega}$ and $r\in (0,\text{diam}\,\Omega)$,
\begin{equation}
                            \label{eq10.19}
\frac {|B_r(x_0)\cap \Omega|} {|B_r(x_0)|}\ge \gamma.
\end{equation}
We begin with the following lemma estimating the second derivatives of solutions.

\begin{lemma}   \label{lem2.4c}
Let $q_0\in (1,\infty)$, $q\in (q_0,\infty)$, $r\in (0,R_0)$, and $\delta\in (0,1)$.
Further, let $u\in W^{1,2}_q(Q_{r}^+)^d$ be a strong solution to \eqref{eq7.41c} in $Q_r^+$ with the boundary conditions \eqref{bdr-cond}, where $p \in W_1^{0,1}(Q_r^+)$, $f \in L_{q_0}(Q_r^+)^d$, and $Dg \in L_{q_0}(Q_{r}^+(z_0))^d$.
Suppose that Assumption \ref{assump1} ($\delta$) holds.
Then we have
\begin{align} \label{eq9.51}
(|D^2u|^{q_0})_{Q_{r/2}^+}^{\frac 1 {q_0}}&\le N(d, \nu, q_0) \Big[ (|Dg|^{q_0})_{Q_{r}^+}^{\frac 1 {q_0}}  +  (|f|^{q_0})_{Q_{r}^+}^{\frac 1 {q_0}} + r^{-1}(|Du-[Du]_{B_r^+}(t)|^{q_0})_{Q_{r}^+}^{\frac 1 {q_0}} \Big] \nonumber \\
&\qquad +N(d,\nu,q_0) r^{-1} \Big[  \sum_{i=1}^{d-1} (|D_d u_i|^{q_0})_{Q_{r}^+}^{\frac 1 {q_0}} +   (|D_{x'} u_d|^{q_0})_{Q_{r}^+}^{\frac 1 {q_0}}  \Big] \nonumber\\
& \qquad +   N(d,\nu,q_0,q) \delta^{\frac 1 {q_0}-\frac 1 {q}}(|D^2u|^q)_{Q_{r}^+}^{\frac 1 {q}},
\end{align}
and
\begin{align} \label{eq9.52}
(|D^2u|^{q_0})_{Q_{r/2}^+}^{\frac 1 {q_0}}&\le N(d, \nu, q_0) \Big[ (|Dg|^{q_0})_{Q_{r}^+}^{\frac 1 {q_0}}  +  (|f|^{q_0})_{Q_{r}^+}^{\frac 1 {q_0}} + r^{-1}(|Du|^{q_0})_{Q_{r}^+}^{\frac 1 {q_0}} \Big]\nonumber\\
& \qquad +   N(d,\nu,q_0,q) \delta^{\frac 1 {q_0}-\frac 1 {q}}(|D^2u|^q)_{Q_{r}^+}^{\frac 1 {q}}.
\end{align}
\end{lemma}
\begin{proof} If $h$ is an integrable function defined on $Q_r^+$, we take the following mollification of $h(t,x)$ for $t \in (-r^2+\varepsilon^2, 0)$:
$$
h^{(\varepsilon)}(t,x) =  \int_{-r^2}^0 h(t+s,x) \eta_\varepsilon(s) \, ds,
$$
where $\eta(t) \in C_0^\infty(\bR)$ with $\eta(t) = 0$ for $t \ge 0$ and $\eta_\varepsilon(t) = \varepsilon^{-2}\eta(t/\varepsilon^2)$.
Note that $h^{(\varepsilon)}(t,x)$ is infinitely differentiable in $t$, and $\partial_t^k h^{(\varepsilon)}(t,x) \in L_{q_0}(Q_{r'}^+)$ for any $k = 1,2,\ldots$ if $h \in L_{q_0}(Q_r^+)$, $r' \in (0,r)$, and $\varepsilon$ is sufficiently small. By mollifying \eqref{eq7.41c} as above with respect to $t$, we have
$$
\partial_t u^{(\varepsilon)} - a_{ij} D_{ij} u^{(\varepsilon)} + \nabla p^{(\varepsilon)} = f^{(\varepsilon)} + \left( a_{ij} D_{ij} u \right)^{(\varepsilon)} - a_{ij} D_{ij} u^{(\varepsilon)}
$$
in $Q_{r'}^+$ for $r' \in (0,r)$.
We see that if we prove the estimate in the lemma for $u^{(\varepsilon)}$, by letting $\varepsilon \to 0$, we obtain the desired estimate for $u$.
Thus, henceforth we assume that $u(t,x)$ is infinitely differentiable in $t$ and $\partial_t^k u, \partial_t^k D_x u, \partial_t^k D_x^2 u \in L_{q_0}(Q_r^+)$ for any $k=1,2,\ldots$.

Let $\zeta_r(x)$ and $\psi_r(t)$ be infinitely differentiable functions defined on $\bR^d$ and $\bR$, respectively, such that
$$
\zeta_r(x) = 1 \quad \text{on} \  B_{2r/3}, \quad \zeta_r(x) = 0 \quad \text{on} \  \bR^d \setminus B_r,
$$
$$
\psi_r(t) = 1 \quad \text{on} \  t \in (-4r^2/9, 4r^2/9), \quad \psi_r(t) = 0 \quad \text{on} \  t \in \bR \setminus (-r^2,r^2).
$$
We set $\phi_r(t,x) = \psi_r(t) \zeta_r(x)$.
Then $\phi_r = 1$ on $Q_{2r/3}$ and $|D\phi_r|\le 4/r$.

For the given $r \in (0, R_0)$, let $\hat{a}_{ij}(t)$ be the matrix defined in Assumption \ref{assump1} ($\delta$) such that
\[
\fint_{Q_r^+} |a_{ij}(t,x)-\hat{a}_{ij}(t)|\,dx\,dt\le \delta, \quad \forall \ i, j =1,2,\ldots, d.
\]
We first consider the following equation:
$$
\left\{
\begin{array}{cccl}
w_t-\hat{a}_{ij}(t)D_{ij} w+\nabla p_1 & = & I_{Q_r^+}(f+(a_{ij}-\hat{a}_{ij})D_{ij} u) & \quad \text{in} \  (-r^2,0)\times \bR^d_+ \\
 \Div  w &= &  (g -  [g(t,\cdot)]_{\zeta_r, B_r^+}) \phi_r &  \quad \text{in} \   (-r^2,0)\times \bR^d_+\\
 w (-r^2, \cdot) & = & 0 & \quad \text{in} \  \bR^d_+
\end{array} \right.
$$
with the Lions boundary conditions
\[
D_d w_k = w_d =0 \quad \text{on} \  \{x_d =0\} \quad \text{for} \  k = 1,2,\ldots, d-1,
\]
 where
$$
[g(t,\cdot)]_{\zeta_r, B_r^+} = \frac{\int_{B_r^+} g(t,y) \zeta_r(y) \, dy}{\int_{B_r^+} \zeta_r(y) \, dy}.
$$
To find a strong solution $(w,p_1)$ to the above equation using Theorem \ref{existence-lemma}, we need to check that
$$
 \tilde{g}(t,x) := \left(g - [g(t,\cdot)]_{\zeta_r, B_r^+}\right) \phi_r \in L_{q_0}\left((-r^2,0)\times \bR^d_+\right), \quad D\tilde{g} \in L_{q_0}\left((-r^2,0)\times \bR^d_+\right),
$$
$\tilde{g}(-r^2,\cdot) = 0$, and that there exists $G=(G_1,\ldots,G_d) \in L_{q_0}(\bR^d_{T,+})$ such that
$$
\partial_t \tilde{g} = \Div G
$$
in $(-r^2,0) \times \bR^d_+$ in the sense as in \eqref{eq0702_01}.
The first three conditions are easy to check, so we check only the last one.
Since $u(t,x)$ is infinitely differentiable in $t$ and $\partial_t \Div u \in L_{q_0}(Q_r^+)$, we have
$$
\partial_t \tilde{g} = \left( \partial_t g - [ \partial_t g(t,\cdot)]_{\zeta_r,B_r^+} \right) \zeta_r(x) \psi_r(t) + \left(g - [g(t,\cdot)]_{\zeta_r, B_r^+}\right) \zeta_r(x) \psi_r'(t),
$$
which belongs to $L_{q_0}\left((-r^2,0) \times \bR^d_+\right)$.
From this it follows that
$$
\int_{B_r^+} \partial_t \tilde{g}(t,x) \, dx = 0.
$$
Then, as in the proof of Theorem \ref{existence-lemma}, we find $G \in W_{q_0}^{0,1}\left((-r^2,0) \times B_r^+\right)$ such that
$$
\left\{
\begin{aligned}
\Div G &= \partial_t \tilde{g}(t,x) \quad \text{in} \  (-r^2,0) \times B_r^+,
\\
G(t,x) &= 0 \quad for \ (t,x) \in (-r^2,0) \times \partial B_r^+.
\end{aligned}
\right.
$$
We again denote by $G$ the zero extension of $G$ on $(-r^2,0) \times \left( \bR^d_+ \setminus B_r^+ \right)$.
Then, using the fact that $\tilde{g}$ has compact support on $(-r^2,0] \times B_r^+$ and the zero boundary condition of $G(t,\cdot)$ on $\partial B_r^+$, we arrive at
$$
\partial_t \tilde{g} = \Div G
$$
in $(-r^2,0) \times \bR^d_+$.
Hence, the existence of $w$ is ensured by Theorem \ref{existence-lemma}.
Further, it follows from Theorem \ref{existence-lemma} that
$$
\|D^2w\|_{L_{q_0}\left((-r^2,0)\times \bR^d_+\right)} \leq N \left( \|f\|_{L_{q_0}(Q_r^+)} +   \|(a_{ij}-\hat{a}_{ij})D_{ij} u\|_{L_{q_0}(Q_r^+)} + \|D\tilde{g}\|_{L_{q_0}\left((-r^2,0)\times \bR^d_+\right)} \right),
$$
where $N = N(d,\nu,q_0)$.
Note that
\[
\begin{split}
\|D\tilde{g}\|_{L_{q_0}\left((-r^2,0)\times \bR^d_+\right)} & =
\left\|D \left[\left(g - [g(t,\cdot)]_{\zeta_r, B_r^+}\right) \phi_r\right]\right\|_{L_{q_0}\left((-r^2,0)\times \bR^d_+\right)}\\
&  \leq  \|Dg\|_{L_{q_0}(Q_r^+)} + \left\|\left(g - [g(t,\cdot)]_{\zeta_r, B_r^+}\right) D \phi_r \right\|_{L_{q_0}(Q_r^+)} \\
&  \leq \|Dg\|_{L_{q_0}(Q_r^+)} + 4r^{-1} \left\|g - [g(t,\cdot)]_{\zeta_r, B_r^+}\right\|_{L_{q_0}(Q_r^+)} ,
\end{split}
\]
where
\[
\begin{split}
\left\|g - [g(t,\cdot)]_{\zeta_r, B_r^+}\right\|_{L_{q_0}(Q_r^+)} & =\frac{1}{\int_{B_r^+} \zeta_r(y) \, dy} \left\| \int_{B_r^+} \left(g(t,x) - g(t,y)\right) \zeta_r(y) \, dy \right\|_{L_{q_0}(Q_r^+)} \\
& \leq N(d) \fint_{B_r^+} \| g(t,x) - g(t,y)\|_{L_{q_0}(Q_r^+)} \zeta_r(y) \, dy
\end{split}
\]
because $\int_{B_r^+} \zeta_r(y) \, dy$ is comparable to $|B_r^+|$. By H\"older's inequality and the Poincar\'e inequality,
\[
\begin{split}
& \fint_{B_r^+} \| g(t,x) - g(t,y)\|_{L_{q_0}(Q_r^+)} \zeta_r(y) \, dy \\
& \leq N\Big(\fint_{B_r^+} \int_{Q_r^+}| g(t,x) - g(t,y)|^{q_0}\,dx \,dt \, dy\Big)^{\frac 1 {q_0}} \le N(d,q_0) r \|Dg\|_{L_{q_0}(Q_r^+)}.
\end{split}
\]
Hence, we obtain
\begin{equation*}
\|D^2w\|_{L_{q_0}((-r^2,0)\times \bR^d_+)}\le N(d,\nu,q_0)\Big[ \|f\|_{L_{q_0}(Q_r^+)} + \|(a_{ij}-\hat{a}_{ij})D_{ij} u\|_{L_{q_0}(Q_r^+)} +   \|Dg\|_{L_{q_0}(Q_{r}^+)} \Big].
\end{equation*}
From this and by using Assumption \ref{assump1} ($\delta$) and H\"{o}lder's inequality for the middle term on the right-hand side of the last estimate, we have
\begin{align}
(|D^2w|^{q_0})_{Q_{r}^+}^{\frac 1 {q_0}}
 \leq N(d, \nu, q_0) \Big[  (|Dg|^{q_0})_{Q_{r}^+}^{\frac 1 {q_0}} + (|f|^{q_0})_{Q_{r}^+}^{\frac 1 {q_0}} \Big]  \label{eq8.37c}
+ N(d,\nu,q_0, q)  \delta^{\frac 1 {q_0}-\frac 1 {q}}(|D^2u|^q)_{Q_{r}^+}^{\frac 1 {q}}.
\end{align}
Now, let $(v,p_2) =(u-w,p-p_1)$. We see that $(v,p_2)$ satisfies
$$
v_t-\bar a_{ij}(t)D_{ij} v+\nabla p_2=0,\quad \Div v=  [g(t,\cdot)]_{\zeta_r, B_r^+}
$$
in $Q_{2r/3}^+$ with the boundary conditions as in \eqref{bdr-cond}.
By using \eqref{eq7.46b} in Lemma \ref{lem1.2} with suitable scaling, we have
\[
\begin{split}
\|D^2v\|_{L_{q_0}(Q_{r/2}^+)} & \le N
\sum_{i=1}^{d-1}r^{-1}\Big(\|D_{x'}v_i-[D_{x'}v_i]_{B_{2r/3}^+}(t)\|_{L_{q_0}(Q_{2r/3}^+)}
+\|D_{d}v_i\|_{L_{q_0}(Q_{2r/3}^+)}\Big)\nonumber\\
&\quad+N r^{-1} \|D_{x'}v_d\|_{L_{q_0}(Q_{2r/3}^+)},
\end{split}
\]
where $N = N(d,\nu,q_0)$.
It is clear that, for instance, $[D_{x'}v_i]_{B_{2r/3}^+}(t)$ can be replaced with $[D_{x'}v_i]_{B_r^+}(t)$ on the right-hand side of the above inequality.
From this, the triangle inequality, and the Poincar\'e inequality on terms involving $w$, we obtain
\begin{align} \nonumber
& (|D^2v|^{q_0})_{Q_{r/2}^+}^{\frac 1 {q_0}} \\ \nonumber
& \le r^{-1} N(d, \nu, q_0) \Big[ (|Du-[Du]_{B_{r}^+}(t)|^{q_0})_{Q_{r}^+}^{\frac 1 {q_0}} + (|Dw-[Dw]_{B_{r}^+}(t)|^{q_0})_{Q_{r}^+}^{\frac 1 {q_0}}  \Big] \\ \nonumber
& \quad + r^{-1} N(d, \nu, q_0)\sum_{i=1}^{d-1} \Big[ (|D_d u_i|^{q_0})_{Q_r^+}^{\frac 1 {q_0}} + (|D_d w_i|^{q_0})_{Q_r^+}^{\frac 1 {q_0}} \Big]\\ \nonumber
& \quad + r^{-1} N(d, \nu, q_0)  \Big[ (|D_{x'} u_d|^{q_0})_{Q_r^+}^{\frac 1 {q_0}} + (|D_{x'} w_d|^{q_0})_{Q_r^+}^{\frac 1 {q_0}} \\  \nonumber
 & \leq N(d, \nu, q_0) r^{-1}  \Big[ (|Du-[Du]_{B_{r}^+}(t)|^{q_0})_{Q_{r}^+}^{\frac 1 {q_0}}  +  \sum_{i=1}^{d-1} (|D_d u_i|^{q_0})_{Q_{r}^+}^{\frac 1 {q_0}} +   (|D_{x'} u_d|^{q_0})_{Q_{r}^+}^{\frac 1 {q_0}} \Big] \\ \label{eq8.43c}
& \quad   + N(d, \nu, q_0)  (|D^2w|^{q_0})_{Q_{r}^+}^{\frac 1 {q_0}}.
\end{align}
Observe that in the last inequality, we applied the Poincar\'e inequality to the terms $D_d w_i$ and $D_{x'} w_d$ with $i =1, 2,\ldots, d-1$ because these terms vanish on $\{x_d =0\}$. Then, by the triangle inequality and \eqref{eq8.43c}, we infer that
\[
\begin{split}
& (|D^2u|^{q_0})_{Q_{r/2}^+}^{\frac 1 {q_0}}  \leq  (|D^2 w|^{q_0})_{Q_{r/2}^+}^{\frac 1 {q_0}} + (|D^2 v|^{q_0})_{Q_{r/2}^+}^{\frac 1 {q_0}} \\
& \leq N(d, \nu, q_0) r^{-1} \Big[ (|Du-[Du]_{B_{r}^+}(t)|^{q_0})_{Q_{r}^+}^{\frac 1 {q_0}}  +  \sum_{i=1}^{d-1} (|D_d u_i|^{q_0})_{Q_{r}^+}^{\frac 1 {q_0}} +   (|D_{x'} u_d|^{q_0})_{Q_{r}^+}^{\frac 1 {q_0}} \Big] \\
& \qquad + N(d, \nu, q_0)(|D^2w|^{q_0})_{Q_{r}^+}^{\frac 1 {q_0}}.
\end{split}
\]
This estimate and \eqref{eq8.37c} imply \eqref{eq9.51} as well as \eqref{eq9.52}.
The proof of the lemma is thus complete.
\end{proof}

\begin{corollary}
                        \label{cor4.8}
Let $q_0\in (1,\infty)$, $q\in (q_0,\infty)$, $r\in (0,R_0)$, and $\delta\in (0,1)$.
Suppose that $T>0$ and $\Omega\subset \bR^d_+$ satisfies \eqref{eq10.19} for some $\gamma>0$. Let $u\in W^{1,2}_q((-T-r^2,0)\times \Omega^r)^d$ be a strong solution to \eqref{eq7.41c} in $(-T-r^2,0)\times \Omega^r$ with the boundary conditions \eqref{bdr-cond} on $(-T-r^2,0) \times \left(\Omega^r \cap \{x : x_d=0\}\right)$, where $p \in W_1^{0,1}((-T-r^2,0)\times \Omega^r)$, $f \in L_{q_0}((-T-r^2,0)\times \Omega^r)^d$, and $Dg \in L_{q_0}((-T-r^2,0)\times \Omega^r)^d$.
 Further, suppose that Assumption \ref{assump1} ($\delta$) holds.
Then we have
\begin{align} \label{eq9.53}
&(|D^2u|^{q_0})_{(-T,0)\times \Omega}^{\frac 1 {q_0}}\le N(d, \nu, q_0,\gamma)
\frac{((T+r^2) |\Omega^r|)^{\frac 1 {q_0}}}{(T|\Omega|)^{\frac 1 {q_0}}}
\Big[ (|Dg|^{q_0})_{(-T-r^2,0)\times \Omega^r}^{\frac 1 {q_0}}  +  (|f|^{q_0})_{(-T-r^2,0)\times \Omega^r}^{\frac 1 {q_0}}\nonumber\\
& \qquad  + r^{-1}(|Du|^{q_0})_{(-T-r^2,0)\times \Omega^r}^{\frac 1 {q_0}} \Big] +   N(d,\nu,q_0,q,\gamma) \frac{((T+r^2) |\Omega^r|)^{\frac 1 {q}}}{(T|\Omega|)^{\frac 1 {q}}} \delta^{\frac 1 {q_0}-\frac 1 {q}}(|D^2u|^q)_{(-T-r^2,0)\times \Omega^r}^{\frac 1 {q}}.
\end{align}
\end{corollary}

\begin{proof}
We use a partition of unity argument. By using \eqref{eq9.52} and the corresponding interior estimate (cf. \cite[Lemma 4.1]{arXiv:1805.04143}), for any $x_0\in \Omega$ and $t_0\in (-T,0)$, we have
\begin{align*}
(|D^2u|^{q_0})_{Q_{r/8}^+(t_0,x_0)}&\le N(d, \nu, q_0) \Big[ (|Dg|^{q_0})_{Q_{r}^+(t_0,x_0)}  +  (|f|^{q_0})_{Q_{r}^+(t_0,x_0)} + r^{-q_0}(|Du|^{q_0})_{Q_{r}^+(t_0,x_0)} \Big]\\
& \qquad +   N(d,\nu,q_0,q) \delta^{1-q_0/q}(|D^2u|^q)_{Q_{r}^+(t_0,x_0)}^{q_0/q}.
\end{align*}
In particular, when $\operatorname{dist}(x_0, \{x_d=0\}) < r/8$ so that we need to apply the boundary estimate \eqref{eq9.52}, we use the relations
$$
Q_{r/8}^+(t_0,x_0) \subset Q_{r/4}^+(t_0,\hat{x}_0) \subset Q_{r/2}^+(t_0,\hat{x}_0) \subset Q_r^+(t_0,x_0),
$$
where $\hat{x}_0$ is the projection of $x_0$ onto $\{x_d=0\}$.

Now to obtain \eqref{eq9.53}, it suffices to integrate both sides of the above inequality with respect to $(t_0,x_0)\in (-T,0)\times \Omega$ and use H\"older's inequality and the interior measure condition \eqref{eq10.19}.
\end{proof}

We now state the following result on the interior mean oscillation estimate of the vorticity solutions, which is \cite[Lemma 4.7]{arXiv:1805.04143}.

\begin{lemma} \label{interior-lemma}
Let $q_1 \in (1, \infty), q_0 \in (1, q_1)$, $\delta \in (0,1)$, $R_0\in (0,1/4)$, $r \in (0, R_0)$, $\kappa \in (0, 1/4)$, and $z_0 \in \overline{Q_1^+}$ such that $Q_r^+(z_0) = Q_r(z_0)$. Suppose that Assumption \ref{assump1} ($\delta$) holds. Let $u\in W^{1,2}_{q_1}(Q_r(z_0))^d$ be a strong solution of \eqref{eq7.41c} in $Q_r(z_0)$, where $p \in W_1^{0,1}(Q_r(z_0))$, $f \in L_{q_0}(Q_r(z_0))^d$, and $Dg \in L_{q_0}(Q_{r}(z_0))^d$.
Then
\begin{align*}
& (|D\omega -(D\omega)_{Q_{\kappa r}(z_0)}|)_{Q_{\kappa r}(z_0)}\\
& \leq N(d, \nu, q_0) \kappa^{-\frac{d+2}{q_0}}(|f|^{q_0})_{Q_{r}(z_0)}^{\frac 1 {q_0}}  + N(d, q_0) \kappa^{-\frac{d+2}{q_0}} (|Dg|^{q_0})_{Q_{r}(z_0)}^{\frac 1 {q_0}} \\
& \quad \quad + N(n, \nu, q_0, q_1)\Big(\kappa^{-\frac{d+2}{q_0}} \delta^{\frac 1 {q_0}-\frac 1 {q_1}} + \kappa\Big) (|D^2u|^{q_1})_{Q_{r}(z_0)}^{1/{q_1}},
\end{align*}
where $\omega = \nabla \times u$ is the matrix of vorticity.
\end{lemma}

In the next lemma, we prove a boundary mean oscillation estimate of the derivatives of the vorticity matrix $\omega = \nabla \times u$.

\begin{lemma} \label{non-vorticity-oss.est}
Let $q_1 \in (1, \infty), q_0 \in (1, q_1)$, $\delta \in (0,1)$, $R_0\in (0,1/4)$, $r \in (0, R_0/4)$, $\kappa \in (0, 1/4)$, and $z_0 \in \overline{Q_1^+}$.
Suppose that Assumption \ref{assump1} ($\delta$) holds. Let $u\in W^{1,2}_{q_1}(Q^+_{5r}(z_0))^d$ be a strong solution to \eqref{eq7.41c} in $Q^+_{5r}(z_0)$ with the boundary conditions \eqref{bdr-cond} on $Q_{5r}^+(z_0) \cap \{(t,x): x_d=0\}$,  where $p \in W_1^{0,1}(Q_{5r}^+(z_0))$,  $f \in L_{q_0}(Q_{5r}^+(z_0))^d$, and $Dg \in L_{q_0}(Q_{5r}^+(z_0))^d$.
Then \begin{align*}
&(|D\omega -(D\omega)_{Q_{\kappa r}^+(z_0 )}|)_{Q_{\kappa r}^+(z_0 )}
\\
& \leq N(d, \nu, q_0) \kappa^{-\frac{d+2}{q_0}}(|f|^{q_0})_{Q_{5r}^+(z_0 )}^{\frac 1 {q_0}}  + N(d, \nu, q_0) \kappa^{-\frac{d+2}{q_0}} (|Dg|^{q_0})_{Q_{5r}^+(z_0 )}^{\frac 1 {q_0}} \\
& \quad \quad + N(n, \nu, q_0, q_1)\Big(\kappa^{-\frac{d+2}{q_0}} \delta^{\frac 1 {q_0}-\frac 1 {q_1}} + \kappa^{\frac 1 2}\Big) (|D^2u|^{q_1})_{Q_{5r}^+(z_0 )}^{1/{q_1}}.
\end{align*}
\end{lemma}

\begin{proof}
We write $z_0 = (t_0, x_0', x_{d0})$, and we split the proof into two cases.\\
\noindent
{\bf Case I}: $x_{d0} \geq r$. In this case, as $Q_{r}^+(z_0) = Q_r(z_0)$, we use Lemma \ref{interior-lemma} to conclude that
\begin{align*}
&(|D\omega -(D\omega)_{Q_{\kappa r}^+(z_0)}|)_{Q_{\kappa r}^+(z_0)}
\\
& \leq N(d, \nu, q_0) \kappa^{-\frac{d+2}{q_0}}(|f|^{q_0})_{Q_{r}^+(z_0)}^{\frac 1 {q_0}}  + N(d, q_0) \kappa^{-\frac{d+2}{q_0}} (|Dg|^{q_0})_{Q_{r}^+(z_0)}^{\frac 1 {q_0}} \\
& \quad \quad + N(n, \nu, q_0, q_1)\Big(\kappa^{-\frac{d+2}{q_0}} \delta^{\frac 1 {q_0}-\frac 1 {q_1}} + \kappa\Big) (|D^2u|^{q_1})_{Q_{r}^+(z_0)}^{1/{q_1}}.
\end{align*}
In addition, observe that since $Q_r^+(z_0) \subset Q_{5r}^+(z_0)$,
\[
(|h|)_{Q_{r}^+(z_0)} \leq N(d)  (|h|)_{Q_{5r}^+(z_0)}
\]
for every measurable function $h$. Therefore, the assertion of the lemma follows.

\noindent
{\bf Case II}: $x_{d0} < r$. In this case, we write $\hat{z}_0 = (t_0, \hat{x}_0)$, where $\hat{x}_0 =( x_0', 0)$. We observe that $Q_{r}^+(z_0) \subset Q_{2r}^+(\hat{z}_0)$. Moreover, as $\kappa < 1/4$ and $|z_0 -\hat{z}_0| <r$, we see that
\[
Q_{\kappa r}^+(z_0) \subset Q_{4r/3}^+(\hat{z}_0) \subset Q_{2r}^+(\hat{z}_0) \subset Q_{4r}^+(\hat{z}_0)  \subset Q_{5r}^+(z_0).
\]
Let $(w, p_1)$ and $(v, p_2)$ be as in the proof of Lemma \ref{lem2.4c}.
In particular, $(w, p_1)$ is the strong solution of
\[
\left\{
\begin{array}{ccc}
  w_t - \hat{a}_{ij}(t) D_{ij} w + \nabla p_1 & = & I_{Q_{4r}(\hat{z}_0)} [f + (a_{ij} - \hat{a}_{ij}(t)) D_{ij} u],   \\
 \Div w &  =  & \phi_{4r}(z-\hat z_0) (g - [g(t,\cdot)]_{\zeta_{4r}(\cdot-\hat x_0),B_{4r}^+(\hat{x}_0)})
 \end{array} \right.
\]
in $(-(4r)^2+t_0, t_0) \times \bR^d_+$ with zero initial condition at $t = t_0-(4r)^2$ and the Lions boundary conditions
\[
D_d w_k = w_d =0  \quad \text{on} \  \{x_d =0\} \quad \text{for} \  k = 1,2,\ldots, d-1.
\]
Here $\hat{a}_{ij}(t)$ is the matrix defined in Assumption \ref{assump1} ($\delta$) such that
\[
\fint_{Q_{4r}^+(\hat z_0)} |a_{ij}(t,x)-\hat{a}_{ij}(t)|\,dx\,dt\le \delta, \quad \forall \ i, j =1,2,\ldots, d.
\]
Moreover, $(v, p_2) = (u-w, p- p_1)$ is a strong solution of
\[
v_t - \bar{a}_{ij}(t) D_{ij} v + \nabla p_2 = 0, \quad \Div v = [g(t,\cdot)]_{\zeta_{4r}(\cdot-\hat x_0),B_{4r}^+(\hat{x}_0)}
\]
in $Q_{8r/3}^+(\hat{z}_0)$ satisfying the Lions boundary conditions on $\{x_d = 0\}$. Let us denote by $\omega_1 = \nabla \times w$ and $\omega_2 = \nabla \times v$ the vorticity matrices of $w$ and $v$, respectively. We deduce from \eqref{eq8.37c} that
\begin{align} \nonumber
& (|D\omega_1|^{q_0})_{Q_{4r}^+(\hat{z}_0)}^{\frac 1 {q_0}}  \leq (|D^2 w|^{q_0})_{Q_{4r}^+(\hat{z}_0)}^{\frac 1 {q_0}}  \label{non-omega-1.est} \\
& \leq  N(d, \nu, q_0)\Big[ (|Dg|^{q_0})_{Q_{4r}^+(\hat{z}_0)}^{\frac 1 {q_0}} +  (|f|^{q_0})_{Q_{4r}^+(\hat{z}_0)}^{\frac 1 {q_0}} \Big] + \delta^{\frac 1 {q_0}-\frac 1 {q_1}}  N(d, \nu, q_1, q_0) (|D^2u|^{q_1})_{Q_{4r}^+(\hat{z}_0)}^{\frac 1 {q_1}}.
\end{align}
By applying Lemma \ref{lem1.3} with $\alpha=1/2$ and suitable scaling, the triangle inequality, and H\"{o}lder's inequality, we obtain
\begin{align*}
& (|D\omega_2 - (D\omega_2)_{Q_{\kappa r}^+(z_0)}|)_{Q_{\kappa r}^+(z_0)}
\\
& \leq N(\kappa r)^{\frac 1 2} [[D\omega_2]]_{C^{\frac 1 4,\frac 1 2}(Q_{4r/3}^+(\hat{z}_0))} \leq N(d, \nu, q_0) \kappa^{\frac 1 2} (|D\omega_2|^{q_0})_{Q_{8r/3}^+(\hat{z}_0)}^{\frac 1 {q_0}}  \\
& \leq  N(d, \nu, q_0) \kappa^{\frac 1 2} \Big[ (|D\omega|^{q_0})_{Q_{4r}^+(\hat{z}_0)}^{\frac 1 {q_0}} + (|D\omega_1|^{q_0})_{Q_{4r}^+(\hat{z}_0)}^{\frac 1 {q_0}}\Big]\\
& \leq N(d, \nu, q_0) \kappa^{\frac 1 2} \Big[ (|D^2 u|^{q_1})_{Q_{4r}^+(\hat{z}_0)}^{\frac 1 {q_1}} + (|D\omega_1|^{q_0})_{Q_{4r}^+(\hat{z}_0)}^{\frac 1 {q_0}}\Big].
\end{align*}
Then, by combining this estimate with \eqref{non-omega-1.est} and the fact that $\delta \in (0,1)$, we infer that
\begin{align}\nonumber
(|D\omega_2 - (D\omega_2)_{Q_{\kappa r}^+(z_0)}|)_{Q_{\kappa r}^+(z_0)} & \leq N(d, \nu, q_0)\kappa^{\frac 1 2} \Big[ (|Dg|^{q_0})_{Q_{4r}^+(\hat{z}_0)}^{\frac 1 {q_0}}  +  (|f|^{q_0})_{Q_{4r}^+(\hat{z}_0)}^{\frac 1 {q_0}} \Big] \\  \label{non-omega-2.est}
& \qquad +   N(d, \nu, q_1, q_0)(\kappa^{\frac 1 2}    + \delta^{\frac 1 {q_0}-\frac 1 {q_1}}) (|D^2u|^{q_1})_{Q_{4r}^+(\hat{z}_0)}^{\frac 1 {q_1}}.
\end{align}
Now, by using the inequality
\[
\fint_{Q_{\kappa r}^+(z_0)}|D\omega -(D\omega)_{Q_{\kappa r}^+(z_0)}| \,dx\, dt \leq 2 \fint_{Q_{\kappa r}^+(z_0)}|D\omega -c| \,dx\,dt
\]
with $c = (D\omega_2)_{Q_{\kappa r}^+(z_0)}$, and then applying
the triangle inequality and H\"{o}lder's inequality, we have
\begin{align*}
&\fint_{Q_{\kappa r}^+(z_0)}|D\omega -(D\omega)_{Q_{\kappa r}^+(z_0)}| \,dx\, dt   \leq 2 \fint_{Q_{\kappa r}^+(z_0)} |D\omega - (D\omega_2)_{Q_{\kappa r}^+(z_0)}| \,dx\, dt\\
& \leq 2  \fint_{Q_{\kappa r}^+(z_0)} |D\omega_2 - (D\omega_2)_{Q_{\kappa r}^+(z_0)}| \,dx\, dt  + N(d, q_0) \kappa^{-\frac{d+2}{q_0}}\left(\fint_{Q_{r}^+(z_0)} |D\omega_1|^{q_0} \,dx\, dt\right)^{\frac 1 {q_0}} \\
& \leq 2  \fint_{Q_{\kappa r}^+(z_0)} |D\omega_2 - (D\omega_2)_{Q_{\kappa r}^+(z_0)}| \,dx\, dt  + N(d, q_0) \kappa^{-\frac{d+2}{q_0}}\left(\fint_{Q_{4r}^+(\hat{z}_0)} |D\omega_1|^{q_0} \,dx\, dt\right)^{\frac 1 {q_0}} .
\end{align*}
This estimate, \eqref{non-omega-1.est}, and \eqref{non-omega-2.est} imply that
\begin{align*}
& (|D\omega -(D\omega)_{Q_{\kappa r}^+(z_0)}|)_{Q_{\kappa r}^+(z_0)}\\
 & \leq N(d, \nu, q_0) \kappa^{-\frac{d+2}{q_0}}(|f|^{q_0})_{Q_{4r}^+(\hat{z}_0)}^{\frac 1 {q_0}} +   N(d, \nu, q_0)\kappa^{-\frac{d+2}{q_0}} (|Dg|^{q_0})_{Q_{4r}^+(\hat{z}_0)}^{\frac 1 {q_0}}   \\
& \quad \quad + N(d, \nu, q_0, q_1)\Big(\kappa^{-\frac{d+2}{q_0}} \delta^{\frac 1 {q_0}-\frac 1 {q_1}}+ \kappa^{\frac 1 2}\Big) (|D^2u|^{q_1})_{Q_{4r}^+(\hat{z}_0)}^{\frac 1 {q_1}}.
\end{align*}
Again, as $Q_{4r}^+(\hat{z}_0) \subset Q_{5r}^+(z_0)$, we see that
\[
(|h|)_{Q_{4r}^+(\hat{z}_0)} \leq N(d) (|h|)_{Q_{5r}^+(z_0)}
\]
for every measurable function $h$, so the assertion of the lemma follows. The proof is then complete.
\end{proof}

Our next lemma gives the key estimates of $D\omega $ and $D^2u$ in the mixed norm.

\begin{lemma}
Let $R\in [1/2,1)$, $R_1\in (0,R_0)$, $\delta \in (0,1)$, $\kappa \in (0, 1/4)$, $s, q \in (1, \infty)$, $q_1 \in (1,\min\{s,q\})$, and $q_0 \in (1, q_1)$.
Assume that Assumption \ref{assump1} ($\delta$) holds. Let $u\in W^{1,2}_{s,q}(Q_{R+R_1}^+)^d$ be a strong solution to \eqref{eq7.41c} in $Q_{R+R_1}^+$ with the boundary conditions \eqref{bdr-cond} on $Q_{R+R_1} \cap \{(t,x): x_d = 0 \}$,  where $p \in W_1^{0,1}(Q_{R+R_1}^+)$, $f \in L_{s,q}(Q_{R+R_1}^+)^d$, and $Dg \in L_{s,q}(Q_{R+R_1}^+)^d$, and let $\omega = \nabla\times u$ denote the matrix of vorticity defined in \eqref{omega-def}.
Then we have
\begin{align}
&\norm{D\omega}_{L_{s,q}(Q_R^+)} \leq N \kappa^{-\frac{d+2}{q_0}}\norm{f}_{L_{s,q}(Q_{R+R_1/2}^+)} + N\kappa^{-\frac{d+2}{q_0}} \norm{Dg}_{L_{s,q}(Q_{R+R_1/2}^+)}  \nonumber\\
                            \label{eq1.39n}
&\quad+N\Big(\kappa^{-\frac{d+2}{q_0}} \delta^{\frac 1 {q_0}-\frac 1 {q_1}} + \kappa^{\frac 1 2}\Big) \norm{D^2u}_{L_{s,q}(Q_{R+R_1/2}^+)}+ N \kappa^{-\frac{d+2}{q_0}} R^{-1}_1\norm{Du}_{L_{s,q}(Q_{R+R_1/2}^+)}
\end{align}
and
\begin{align}
&\norm{D^2u}_{L_{s,q}(Q_R^+)} \leq N \kappa^{-\frac {d+2} {q_0}}\norm{f}_{L_{s,q}(Q_{R+R_1}^+)} +N \kappa^{-\frac {d+2} {q_0}} \norm{Dg}_{L_{s,q}(Q_{R+R_1}^+)}   \nonumber\\
                            \label{eq1.40n}
&\quad+N\Big(\kappa^{-\frac {d+2} {q_0}} \delta^{\frac 1 {q_0}-\frac 1 {q_1}} + \kappa^{\frac 1 2}\Big) \norm{D^2u}_{L_{s,q}(Q_{R+R_1}^+)}+
N \kappa^{-\frac{d+2}{q_0}} R^{-1}_1\norm{Du}_{L_{s,q}(Q_{R+R_1}^+)}.
\end{align}
\end{lemma}
\begin{proof}
We first prove \eqref{eq1.39n}.
We consider two cases.
\\ \noindent
{\bf Case I}: $r \in (0, R_1/10)$.
It follows from Lemma \ref{non-vorticity-oss.est} that for all $z_0 \in \overline{Q_R^+}$,
\begin{align*}
&(|D\omega -(D\omega)_{Q_{\kappa r}^+(z_0)}|)_{Q_{\kappa r}^+(z_0)} \leq
N(d, \nu, q_0) \kappa^{-\frac{d+2}{q_0}} (|f|^{q_0})^{\frac 1 {q_0}}_{Q_{5r}^+(z_0)} \\
& + N(d, \nu, q_0) \kappa^{-\frac{d+2}{q_0}} (|Dg|^{q_0})^{\frac 1 {q_0}}_{Q_{5r}^+(z_0)}
+ N(d, \nu, q_0, q_1)\Big(\kappa^{-\frac{d+2}{q_0}} \delta^{\frac 1 {q_0}-\frac 1 {q_1}}+ \kappa^{\frac 1 2}\Big)(|D^2u|^{q_1})^{\frac 1 {q_1}}_{Q_{5r}^+(z_0)}.
\end{align*}
Observe that because $r < R_1/10$, we have $Q_{5r}^+(z_0) \subset Q_{R+R_1/2}^+$. Therefore,
\[
\begin{split}
&  (|Dg|^{q_0})^{\frac 1 {q_0}}_{Q_{5r}^+(z_0)} \leq \mathcal{M}(I_{Q_{_{R+R_1/2}}^+}|Dg|^{q_0})^{\frac 1 {q_0}}(z_0), \\
& (|f|^{q_0})^{\frac 1 {q_0}}_{Q_{5r}^+(z_0)} \leq \mathcal{M}(I_{Q_{_{R+R_1/2}}^+}|f|^{q_0})^{\frac 1 {q_0}}(z_0),  \quad  \text{and} \\
& (|D^2u|^{q_1})^{\frac 1 {q_1}}_{Q_{5r}^+(z_0)}  \le \mathcal{M}(I_{Q_{_{R+R_1/2}}^+}|D^2u|^{q_1})^{\frac 1 {q_1}}(z_0),
\end{split}
\]
where $\mathcal{M}$ is the Hardy--Littlewood maximal function. These estimates imply that
\begin{align*}
&(|D\omega -(D\omega)_{Q_{\kappa r}^+(z_0)}|)_{Q_{\kappa r}^+(z_0)}
 \leq  N \kappa^{-\frac{d+2}{q_0}} \mathcal{M}(I_{Q_{_{R+R_1/2}}^+}|f|^{q_0})^{\frac 1 {q_0}}(z_0)\\
&\quad\quad +N \kappa^{-\frac{d+2}{q_0}} \mathcal{M}(I_{Q_{_{R+R_1/2}}^+}|Dg|^{q_0})^{\frac 1 {q_0}}(z_0)
+ N\Big(\kappa^{-\frac{d+2}{q_0}} \delta^{\frac 1 {q_0}-\frac 1 {q_1}}+ \kappa^{\frac 1 2}\Big)\mathcal{M}(I_{Q_{_{R+R_1/2}}^+}|D^2u|^{q_1})^{\frac 1 {q_1}}(z_0).
\end{align*}
\noindent
{\bf Case II}: $r\in [R_1/10, 2R/\kappa)$ and $z_0 = (t_0,x_0) \in \overline{Q_R^+}$ such that $t_0 \in [-R^2 + (\kappa r)^2/2, 0]$. In this case, we simply estimate
\begin{align}
& (|D\omega -(D\omega)_{Q_{\kappa r}^+(z_0)\cap Q_R^+}|)_{Q_{\kappa r}^+(z_0)\cap Q_R^+}
\le 2(|D\omega|)_{Q_{\kappa r}^+(z_0)\cap Q_R^+}
\le 2(|D\omega|^{q_0})^{\frac 1 {q_0}}_{Q_{\kappa r}^+(z_0)\cap Q_R^+}\nonumber\\
&\le N\kappa^{-\frac{d+2}{q_0}} \Big[ (|f|^{q_0})_{Q_{\kappa r+R_1/2}^+(z_0)\cap Q_{R+R_1/2}^+}^{\frac 1 {q_0}}
+(|Dg|^{q_0})_{Q_{\kappa r+R_1/2}^+(z_0)\cap Q_{R+R_1/2}^+}^{\frac 1 {q_0}} \label{eq9.19}
\\   & \qquad  + R^{-1}_1(|Du|^{q_0})_{Q_{\kappa r+R_1/2}^+(z_0)\cap Q_{R+R_1/2}^+}^{\frac 1 {q_0}} \Big] +   N
\kappa^{-\frac{d+2}{q_1}} \delta^{\frac 1 {q_0}-\frac 1 {q_1}}(|D^2u|^{q_1})_{Q_{\kappa r+R_1/2}^+(z_0)\cap Q_{R+R_1/2}^+}^{\frac 1 {q_1}},\nonumber
\end{align}
where we used Corollary \ref{cor4.8} and $R_1/10\le r$ in the last inequality.

Now, we take $\cX=Q_R^+$ and define the dyadic sharp function $(D\omega)^\#_{\text{dy}}$ of $D\omega$ in $\cX$. From the above two cases, we conclude that for any $z_0\in \cX$,
\begin{align*}
&(D\omega)_{\text{dy}}^{\#}(z_0)    \leq N \kappa^{-\frac{d+2}{q_0}} \Big[ \mathcal{M}(I_{Q_{R+R_1/2}^+}(|f|+|Dg|)^{q_0})^{\frac 1 {q_0}}(z_0)+R_1^{-1}\mathcal{M}(I_{Q_{R+R_1/2}^+}|Du|^{q_0})^{\frac 1 {q_0}}(z_0)\Big] \\
& \quad + N\Big(\kappa^{-\frac{d+2}{q_0}} \delta^{ \frac 1 {q_0}-\frac 1 {q_1}}+ \kappa^{\frac 1 2}\Big) \mathcal{M}(I_{Q_{_{R+R_1/2}}^+}|D^2u|^{q_1})^{\frac 1 {q_1}} (z_0).
\end{align*}
Indeed, by the properties in \cite[Theorem 2.1]{MR3812104}, for any element $Q_\alpha^n$ in the partitions of $Q_R^+$, there exists $z_0 = (t_0,x_0)$ such that $-R^2 + (\kappa r)^2/2 \leq t_0 \leq 0$ and
$$
Q_\alpha^n \subset Q_{\kappa r}^+(z_0) \cap Q_R^+,
$$
where the volumes of $Q_\alpha^n$ and $Q_{\kappa r}^+(z_0)$ are comparable. Recalling that $1<q_0<q_1<\min\{s,q\}$, by Lemma \ref{mixed-norm-lemma} and the Hardy--Littlewood maximal function theorem in mixed-norm spaces (see, for instance, \cite[Corollary 2.6]{MR3812104}),
\begin{align*}
&\norm{D\omega}_{L_{s,q}(Q_R^+)}
\leq N\Big[ \norm{(D\omega)_{\text{dy}}^{\#}}_{L_{s,q}
(Q_R^+)} + R^{\frac 2 s+\frac d q}(|D\omega|)_{Q_R^+} \Big]\\
& \leq  N \kappa^{-\frac{d+2}{q_0}} \norm{\mathcal{M}(I_{Q_{_{R+R_1/2}}^+}(|f|+|Dg|)^{q_0})^{\frac 1 {q_0}}}_{L_{s,q}(\bR^{d+1})} +N\kappa^{-\frac{d+2}{q_0}} R_1^{-1}\norm{\mathcal{M}(I_{Q_{_{R+R_1/2}}^+}|Du|^{q_0})^{\frac 1 {q_0}}}_{L_{s,q}(\bR^{d+1})}\\
&\quad+N\Big(\kappa^{-\frac{d+2}{q_0}} \delta^{\frac 1 {q_0}-\frac 1 {q_1}}+ \kappa^{\frac 1 2}\Big) \norm{\mathcal{M}(I_{Q_{_{R+R_1/2}}^+}|D^2u|^{q_1})^{\frac 1 {q_1}}}_{L_{s,q}(\bR^{d+1})} + NR^{\frac 2 s+\frac d q} (|D\omega|)_{Q_{R}^+}\\
&\le N \Big[ \kappa^{-\frac{d+2}{q_0}}\norm{f}_{L_{s,q}(Q_{R+R_1/2}^+)} + \kappa^{-\frac{d+2}{q_0}} \norm{Dg}_{L_{s,q}(Q_{R+R_1/2}^+)}+\kappa^{-\frac{d+2}{q_0}} R_1^{-1}\norm{Du}_{L_{s,q}(Q_{R+R_1/2}^+)}  \\
& \quad +\Big(\kappa^{-\frac{d+2}{q_0}} \delta^{\frac 1 {q_0}-\frac 1 {q_1}} + \kappa^{\frac 1 2}\Big) \norm{D^2u}_{L_{s,q}(Q_{R+R_1/2}^+)}+  R^{\frac 2 s+\frac d q} (|D\omega|)_{Q_{R}^+}\Big].
\end{align*}
Similar to \eqref{eq9.19}, by Corollary \ref{cor4.8}, the last term on the right-hand side above is bounded by
\begin{align*}
&NR^{\frac 2 s+\frac d q}\Big[ (|f|^{q_0})_{Q_{R+R_1/2}^+}^{\frac 1 {q_0}}
+(|Dg|^{q_0})_{Q_{R+R_1/2}^+}^{\frac 1 {q_0}}
+ R^{-1}_1(|Du|^{q_0})_{Q_{R+R_1/2}^+}^{\frac 1 {q_0}}+\delta^{\frac 1 {q_0}-\frac 1 {q_1}}(|D^2u|^{q_1})_{Q_{R+R_1/2}^+}^{\frac 1 {q_1}}\Big] \\
&\le N\Big[\norm{f}_{L_{s,q}(Q_{R+R_1/2}^+)} + \norm{Dg}_{L_{s,q}(Q_{R+R_1/2}^+)}
+R_1^{-1}\norm{Du}_{L_{s,q}(Q_{R+R_1/2}^+)}
+\delta^{\frac 1 {q_0}-\frac 1 {q_1}}\norm{D^2u}_{L_{s,q}(Q_{R+R_1/2}^+)}\Big],
\end{align*}
where we used H\"older's inequality in the last line.
Combining the two inequalities above, we obtain \eqref{eq1.39n}.

Next, we prove \eqref{eq1.40n}.
Since $u$ satisfies \eqref{eq10.53} in $Q_{R+R_1}^+$ with $g$ in place of $\tilde g$ and with either the zero Dirichlet or Neumann boundary condition, by the boundary mixed-norm Sobolev estimate for non-divergence form elliptic equations (cf. \cite{MR3812104}), we have
$$
\|D^2u \|_{L_{s,q}(Q_{R}^+)}
\le N\|D\omega\|_{L_{s,q}(Q_{R+R_1/2}^+)} + N\|Dg\|_{L_{s,q}(Q_{R+R_1/2}^+)}  +N R^{-2}_1\|u\|_{L_{s,q}(Q_{R+R_1/2}^+)}.
$$
Replacing $u_i$ with $u_i-[u_i]_{B_{R+R_1/2}^+}(t)$ for $i=1,\ldots, d-1$ and using the interior and boundary Poincar\'e inequality, we infer that
\begin{equation}
                           \label{eq1.52n}
\|D^2u \|_{L_{s,q}(Q_{R}^+)}
\le N\|D\omega\|_{L_{s,q}(Q_{R+R_1/2}^+)} + N\|Dg\|_{L_{s,q}(Q_{R+R_1/2}^+)}+N R^{-1}_1\|Du\|_{L_{s,q}(Q_{R+R_1/2}^+)}.
\end{equation}
Combining \eqref{eq1.52n} and \eqref{eq1.39n} with $R+R_1/2$ in place of $R$, we obtain \eqref{eq1.40n}.
The lemma is proved.
\end{proof}

Now we are ready to give the proof of Theorem \ref{thm2.3b}.

\begin{proof}[Proof of Theorem \ref{thm2.3b}]
For $k=1,2,\ldots$, we denote $Q^k=(-(1-2^{-k})^2,0)\times B_{1-2^{-k}}^+$. Let $k_0$ be the smallest positive integer such that $2^{-k_0-1}\le R_0$. For $k\ge k_0$, we apply \eqref{eq1.40n} with $R=1-2^{-k}$ and $R_1=2^{-k-1}$ to get
\begin{align}
&\norm{D^2u}_{L_{s,q}(Q^{k})} \leq N \kappa^{-\frac {d+2} {q_0}}\norm{f}_{L_{s,q}(Q^{k+1})} + N \kappa^{-\frac {d+2} {q_0}}\norm{Dg}_{L_{s,q}(Q^{k+1})} \nonumber\\
                            \label{eq2.15n}
&\quad+N\Big(\kappa^{-\frac {d+2} {q_0}} \delta^{\frac 1 {q_0}-\frac 1 {q_1}} + \kappa^{\frac 1 2}\Big) \norm{D^2u}_{L_{s,q}(Q^{k+1})}+
N \kappa^{-\frac{d+2}{q_0}} 2^{k}\norm{Du}_{L_{s,q}(Q^{k+1})}.
\end{align}
From \eqref{eq2.15n} and the interpolation inequalities, we obtain
\begin{align}
&\norm{D^2u}_{L_{s,q}(Q^{k})} \leq N \kappa^{-\frac {d+2} {q_0}}\norm{f}_{L_{s,q}(Q^{k+1})} + N \kappa^{-\frac {d+2} {q_0}}\|Dg\|_{L_{s,q}(Q^{k+1})} \nonumber\\
                            \label{eq3.21}
&\quad+N\Big(\kappa^{-\frac {d+2} {q_0}} \delta^{\frac 1 {q_0}-\frac 1 {q_1}} + \kappa^{\frac 1 2}\Big) \norm{D^2u}_{L_{s,q}(Q^{k+1})}+
N \kappa^{-\frac 1 2-\frac {2(d+2)} {q_0}}2^{2k}\norm{u}_{L_{s,q}(Q^{k+1})},
\end{align}
where the constants $N$ above are independent of $k$.
We then take $\kappa$ sufficiently small and then $\delta$ sufficiently small so that
$$
N\Big(\kappa^{-\frac {d+2} {q_0}} \delta^{\frac 1 {q_0}-\frac 1 {q_1}} + \kappa^{\frac 1 2}\Big)\le 1/5.
$$
Finally, we multiply both sides of \eqref{eq3.21} by $5^{-k}$ and sum over $k=k_0,k_0+1,\ldots$ to obtain the desired estimate. The theorem is proved.
\end{proof}

\bibliographystyle{plain}

\def\cprime{$'$}

\end{document}